\documentclass{article}

\usepackage{amssymb,amsmath,mathrsfs,amsthm}
\usepackage{mathtools}
\usepackage{geometry}
\usepackage{amsfonts}
\usepackage{esint,comment}
 \newtheorem{theorem}{Theorem}[section]
 \newtheorem{corollary}[theorem]{Corollary}
 \newtheorem{lemma}[theorem]{Lemma}
 
 \newtheorem{proposition}[theorem]{Proposition}
 \theoremstyle{definition}
 \newtheorem{definition}[theorem]{Definition}
 \theoremstyle{remark}

 \numberwithin{equation}{section}

\def \no#1#2#3 {{\bf #1} (#3), #2.}
\def \eds#1#2#3 {#1, #2, #3.}

\def\R{{\mathbb R}}
\def\e{{\rm e}}
\def\d{{\rm d}}

\def\Y{{\rm Y}}

\def\T{{\mathbb T}}

\def\eps{{\varepsilon}}
\def\:{{\colon}}

\def\be#1{\begin{equation}\label{#1}}
\def\ee{\end{equation}}

\def\<{\langle}
\def\>{\rangle}
\def\coloneqq{:=}

\newcommand{\p}{\partial}
\newcommand{\lewy}{\left\lbrace}
\newcommand{\prawy}{\right\rbrace}

\newcommand{\TT}{\mathbb{T}}
\newcommand{\ZZ}{\mathbb{Z}}
\newcommand{\NN}{\mathbb{N}}
\newcommand{\RR}{\mathbb{R}}
\newcommand{\eqnb}{\begin{equation}}
\newcommand{\eqne}{\end{equation}}

\newcommand\blfootnote[1]{%
  \begingroup
  \renewcommand\thefootnote{}\footnote{#1}%
  \addtocounter{footnote}{-1}%
  \endgroup
}

\begin{document}
\title{Partial regularity for a surface growth model}
\author{Wojciech S. O\.za\'nski, James C. Robinson}
\date{\vspace{-5ex}}
\maketitle
\blfootnote{\noindent Mathematics Institute, Zeeman Building, University of Warwick, Coventry CV4 7AL, UK\\ w.s.ozanski@warwick.ac.uk, j.c.robinson@warwick.ac.uk}




\begin{abstract}
We prove two partial regularity results for the scalar equation $u_t+u_{xxxx}+\partial_{xx}u_x^2=0$, a model of surface growth arising from the physical process of molecular epitaxy. We show that the set of space-time singularities has (upper) box-counting dimension no larger than $7/6$ and $1$-dimensional (parabolic) Hausdorff measure zero. These parallel the results available for the three-dimensional Navier--Stokes equations. In fact the mathematical theory of the surface growth model is known to share a number of striking similarities with the Navier--Stokes equations, and the partial regularity results are the next step towards understanding this remarkable similarity. As far as we know the surface growth model is the only lower-dimensional ``mini-model'' of the Navier--Stokes equations for which such an analogue of the partial regularity theory has been proved.
In the course of our proof, which is inspired by the rescaling analysis of Lin (1998) and Ladyzhenskaya \& Seregin (1999), we develop certain \emph{nonlinear parabolic Poincar\'e inequality}, which is a concept of independent interest. We believe that similar inequalities could be applicable in other parabolic equations.
\end{abstract}


\section{Introduction}

In this paper we  consider the one-dimensional model of surface growth
\begin{equation}\label{BR}
u_t+u_{xxxx}+\partial_{xx}u_x^2=0
\end{equation}
on the one-dimensional torus $\T$, under the assumption that $\int_{\T}u=0$; we refer to this in what follows as the SGM.

As previously observed by Bl\"omker \& Romito \cite{blomkerromito09,blomkerromito12}, this model shares many striking similarities with the three-dimensional Navier--Stokes equations. In particular, in their 2009 paper Bl\"omker \& Romito proved local existence in the critical space $\dot H^{1/2}$ and (spatial) smoothness for solutions bounded in $L^{8/(2\alpha-1)}((0,T);H^\alpha)$ for all $1/2<\alpha<9/2$; in the 2012 paper they prove local existence in a critical space of a similar type to that occurring in the paper by Koch \& Tataru \cite{koch_tataru} for the Navier--Stokes equations.

The aim of this paper is to prove partial regularity results for \eqref{BR} that are analogues of those proved by Caffarelli, Kohn, \& Nirenberg \cite{CKN} for the Navier--Stokes equations. Perhaps surprisingly their inductive method does not seem well adapted to \eqref{BR}, and instead we use the rescaling approach of Lin \cite{lin} and Ladyzhenskaya \& Seregin \cite{ladyzhenskaya_seregin}. The main issue is that the biharmonic heat kernel, given in the one-dimensional case by
$$
K(x,t)=\alpha t^{-1/4}f(|x|t^{-1/4}),\qquad\mbox{where}\qquad f(x)=\int_0^\infty \e^{-s^4}\cos(xs)\,\d s
$$
and $\alpha$ is a normalising constant, takes negative values so cannot be used as the basis of the construction of a suitable sequence of test functions for use in the local energy inequality.

Our main result is the following.
\begin{theorem}
Let $u$ be a suitable weak solution of the surface growth model. Then
\begin{enumerate}
\item[(i)] there exist $\varepsilon_0,R_0>0$ such that if $r<R_0$ and \[
\frac{1}{r^2}\int_{Q(z,r)}|u_x|^3\leq \varepsilon_0
\]
for a cylinder $Q(z,r)$, then $u$ is H\"older continuous in $Q(z,r/2)$;
\item[(ii)] there exists $\varepsilon_1>0$ such that if either
\[
  \limsup_{r\to0}\frac{1}{r}\int_{Q(z,r)}u_{xx}^2\leq \varepsilon_1\quad \text{ or }  \quad\limsup_{r\to 0}\left\{\sup_{t-r^4\le s\le t+r^4}\frac{1}{r}\int_{B_r(x)}u(s)^2 \right\}\leq \varepsilon_1
\]
  then $u$ is H\"older continuous in $Q(z,\rho )$ for some $\rho >0$. 
\end{enumerate}
\end{theorem}
Here $Q(z,r)$ denotes the parabolic cylinder of radius $r$ centred at $z$ (see Section \ref{sec_notation} below), and the notion (and existence) of suitable weak solutions is discussed in Section \ref{sec_suitable_weak_solutions}. Using these conditional regularity results we deduce upper bounds on the dimension of the space-time singular set, which we take here to be
\eqnb\label{singset}\begin{split}
S&=\{(x,t)\in\T\times(0,\infty):\ u\mbox{ is not space-time H\"older continuous} \\
&\hspace{4.5cm}\qquad\qquad\qquad\mbox{in any neighbourhood of }(x,t)\}.
\end{split}\eqne
Note that $S$ is closed. Namely we show that for every compact $K\subset \TT\times (0,\infty )$
\[
d_B(S\cap K) \leq 7/6\quad \text{ and }\quad  \mathcal{P}^1(S) =0,
\]
see Corollary \ref{cor_bound_on_db} and Corollary \ref{cor_bound_on_dH}, respectively. Here $d_B$ denotes the (upper) box-counting dimension (see \eqref{minkowski_def} for the definition) and $\mathcal{P}^1$ denotes the $1$-dimensional parabolic Hausdorff measure appropriate for scaling of the equations (see \eqref{def_of_Pk} for details). The point of considering the intersection $S\cap K$ is to separate $S$ from the set $\{ (x,0)\colon x\in \TT \}$. This is a technical matter related to the definition of the box-counting dimension. Indeed, in order to deduce the bound from part (i) of the theorem above one first needs to guarantee that $Q(z,r)\subset \TT\times (0,\infty )$ for sufficiently small $r$, uniformly in $z\in S$ (see the proof of Corollary \ref{cor_bound_on_db}), and we overcome this problem by intersecting $S$ with a compact set. (A similar issue appears in the case of the Navier--Stokes equations, see Theorem 15.8 in Robinson et al. \cite{NSE_book}.) This issue does not appear in the second estimate, $\mathcal{P}^1(S)=0$, and in the case when the initial condition of a suitable weak solution is sufficiently regular (say $H^{1/2}$) as then a weak-strong uniqueness result (see Theorem 2.11 in Bl\"omker \& Romito \cite{blomkerromito09}, for example) guarantees that $u$ is smooth for small times. 

The estimate $\mathcal{P}^1(S)=0$ implies that $d_H(S)$, the Hausdorff dimension of $S$, is no greater than $1$. In the case of the Navier--Stokes equations, the corresponding results are $d_B(S\cap K)\leq 5/3$ for any compact set $K$ and $\mathcal{P}^1_{NSE } (S) =0$ (see, for example, Chapters 15 and 16 in \cite{NSE_book}), where $\mathcal{P}^1_{NSE}$ is the $1$-dimensional parabolic Hausdorff measure which respects the Navier--Stokes scaling in $\RR^3\times \RR$ (hence the subscript ``NSE''). In the case of the Navier--Stokes equations the bound on the box-counting dimension has been improved, the current sharpest bound being $d_B(S\cap K)\leq 2400/1903  (\approx 1.261) $, due to He et al. \cite{he_wang_zhou_2017}.\\

As for the definition \eqref{singset}, note that if $u$ is spatially H\"older continuous on $\T$ with some exponent $\theta$ then $u\in H^\alpha(\T)$ for all $0<\alpha<\theta$, using the Sobolev--Slobodeckii characterisation of $H^\alpha(\T)$ as the collection of all functions such that
$$
\int_\T\int_\T\frac{|f(x)-f(y)|^2}{|x-y|^{1+2\alpha}}\,\d x\,\d y<\infty
$$
(see Di Nezza, Palatucci, \& Valdinoci \cite{NePaVa:2012}); it follows (using arguments from Bl\"omker \& Romito \cite{blomkerromito09}) that if $u$ is space-time H\"older continuous on $[0,T]\times\T$ then $u$ is spatially smooth. However, the condition $u\in L^\infty_tL^\infty_x$ (which is the SGM equivalent of the $L^\infty_tL^3_x$ regularity for the Navier--Stokes equations, see Escauriaza, Seregin, \& S\v{v}er\'ak \cite{ESS_2003}) is not yet known to be sufficient for the regularity of the SGM. This is why we do not use local essential boundedness in our definition of $S$. Observe also that $S$ is closed (as its complement is open in $\TT\times (0,\infty)$).

Note that it is not entirely clear whether or not the definition in \eqref{singset} is the correct one for the SGM, since a \textit{local} conditional regularity result that guarantees spatial smoothness under a localised H\"older condition of $u$ is currently unknown. However, a closely related result has recently been proved by O\.za\'nski \cite{O_Serrin_cond}: if $u_x \in L_t^{q'}L_x^q (Q)$, where $Q$ is a cylinder and $q',q\in (1,\infty ]$ are such that $4/q'+1/q\leq 1$, then $u\in C^\infty (Q)$. This can be thought of as an analogue of the local Serrin condition in the Navier--Stokes equation, which guarantees that a weak solution $u$ satisfying $u\in L_t^{q'}L_x^q (Q) $ for $2/q'+3/q=1$ is smooth in space, see Section 8.5 in Robinson et al. \cite{NSE_book}, for example.  

The structure of the article is as follows. In the remainder of this section we introduce some notation, in Section \ref{sec_suitable_weak_solutions} we introduce the notion of suitable weak solutions and we show global-in-time existence of such solutions for any initial condition $u_0\in L^2$ with zero mean. In Section \ref{sec_PPI} we introduce a ``nonlinear parabolic Poincar\'e inequality'', which is vital for both of our partial regularity results and a concept of independent interest. We then prove two local regularity results for the surface growth model, the first in terms of $u_x$ (Section \ref{sec_1st_local_reg}) and the second one in terms of $u_{xx}$ (Section \ref{sec_2nd_local_reg}). As a consequence we can show that the (upper) box-counting dimension of the space-time singular set is no larger than $7/6$, and that its one-dimensional parabolic Hausdorff measure is zero.

\subsection{Notation}\label{sec_notation}
With $z=(x,t)$ we define the centred\footnote{Note that in many papers $Q(z,r)$ is used for the `non-anticipating' cylinder which in this case would be $B_r(x)\times(t-r^4,t)$.} parabolic cylinder $Q(z,r)$ to be
$$
Q(z,r)=(x-r,x+r)\times (t-r^4,t+r^4).
$$
Note that the `cylinder' here is in fact a rectangle. We often use the notation $Q_r$ for a cylinder $Q(z,r)$ for some $z$. Set
\eqnb\label{notation_mean_over_cylinder}
f_{z,r}\coloneqq \fint_{Q(z,r)}f=\frac{1}{|Q(z,r)|}\int_{Q(z,r)}f.
\eqne
We set $L^2=L^2(\TT )$, $H^s=H^s (\TT )$, and $W^{s,p}= W^{s,p} (\TT )$ ($s\geq 0$, $p\geq 1$), function spaces consisting of periodic functions: for example $W^{s,p}$ is the completion of the space of smooth and periodic functions on $\TT $ in the $W^{s,p}$ norm. The norm on $H^s$ is equivalent to
\[
\left(\sum_{k\in \ZZ} \left( 1 + |k|^{2s} \right) |\hat{f}(k) |^2\right)^{1/2},
\]
where $\hat{f} (k)$ denotes the $k$-th Fourier coefficient of $f$. We write $\| \cdot \|$ to denote the $L^2$ norm and we write a dot ``$\cdot$'' above a function space to denote the closed subspace of functions with zero integral so that, for example,
\[
\dot{H}^s \coloneqq \lewy f\in H^s \, : \, \int_\T f =0 \prawy , \quad s\geq 0.
\]
We will also write 
\[\| f \|_{\dot{H}^s} = \left(\sum_{k\in \ZZ}  |k|^{2s} |\hat{f}(k) |^2\right)^{1/2} = c\| \p_x^s f \|
\] to denote the $H^s$ seminorm (where the last equality holds for integer $s$). Note that if $u\in \dot{H}^s$ then
$$
\|u\|_{\dot{H}^{s'}}\le\|u\|_{\dot{H}^s}\qquad\mbox{for all }s'\le s
$$
and hence that $\|u\|_{H^{s'}}\le c_s\|\partial_x^s u\|$ if $s$ is an integer. We will use the Sobolev interpolation,
\begin{equation}\label{Sobinterp}
\|u\|_{H^s}\le\|u\|_{H^{s_1}}^\theta\|u\|_{H^{s_2}}^{1-\theta}
\end{equation}
and a similar inequality for the seminorms, where $s_1\le s\le s_2$ and $s=\theta s_1+(1-\theta)s_2$.

We write $\int \coloneqq \int_{\TT}$ and, given $T>0$, we denote the space of smooth functions that are periodic with respect to the spatial variable and compactly supported in a time interval $I$ by $C_0^\infty (\TT \times I)$. We denote any universal constant by a $C$ or $c$. 

\section{Suitable weak solutions}\label{sec_suitable_weak_solutions}
We first define the notion of a weak solution of the problem \eqref{BR}.

\begin{definition}[Weak solution]\label{def_of_weak_sol}
We say that $u$ is a (global-in-time) \emph{weak solution} of the surface growth initial value problem
\begin{equation}\label{SGM}
\begin{cases}
u_t =-u_{xxxx} - \p_{xx}u_x^2,\\
u(0)=u_0\in\dot L^2,
\end{cases}
\end{equation}
if for every $T>0$
\eqnb\label{weak_sol_regularity}
u \in L^\infty ((0,T);\dot{L}^2) \cap L^2 ((0,T);\dot{H}^2)
\eqne
and
\eqnb\label{weak_sol_distr_eq}
- \int_0^T \int \left( u \, \phi_t - u_{xx} \phi_{xx} - u_x^2 \phi_{xx} \right) =\int u_0 \phi (0)
\eqne
for all $\phi \in C_0^\infty ({\TT} \times [0,T) )$.
\end{definition}

Note that a simple procedure of cutting off $\phi$ in time (and an application of the Lebesgue Differentiation Theorem) gives that \eqref{weak_sol_distr_eq} is equivalent to
\eqnb\label{weak_sol_distr_eq_alt}
\int u(t) \phi(t) - \int_s^t \int \left( u \, \phi_t - u_{xx} \phi_{xx} - u_x^2 \phi_{xx} \right) =\int  u(s)\phi(s)
\eqne
being satisfied for all $\phi \in C_0^\infty ({\TT} \times [0,T) )$ and almost all $s,t$ with $0\leq s <t$ (including $s=0$, in which case $u(0)= u_0$).

Note also that it follows from the regularity \eqref{weak_sol_regularity} enjoyed by any weak solution  that
\eqnb\label{ux_is_in_L10/3}
u_x \in L^{10/3} ((0,T);L^{10/3}).
\eqne
Indeed, using Sobolev interpolation \eqref{Sobinterp}, for any $0\le s\le 2$ we have
$$
\|u\|_{H^s}\le\|u\|_{L^2}^{1-s/2}\|u\|_{H^2}^{s/2},
$$
and so the 1D embedding $H^s\subset L^p$ when $s=1/2-1/p$ gives
$$
\|u_x\|_{L^p}\le\|u\|_{L^2}^{(2+p)/4p}\|u\|_{H^2}^{(3p-2)/4p},
$$
and so $u_x\in L^{8p/(3p-2)}((0,T);L^p)$; in particular $u_x\in L^{10/3}((0,T);L^{10/3})$.

We now briefly recall the proof of the existence of global-in-time weak solutions to the surface growth initial value problem for any initial data $u_0 \in \dot{L}^2$. We give a sketch of the proof (due to Stein \& Winkler \cite{SteWin:05}) since it will be required in showing the local energy inequality (Theorem \ref{thm:suitable}).

\begin{theorem}[Existence of weak solutions]\label{thm_existence_of_weak_sols}
For each $u_0\in\dot L^2$ there exists at least one weak solution of the surface growth initial value problem \eqref{SGM}.
\end{theorem}


\begin{proof}[Proof (sketch).]
Fix $T>0$ and take $N\in \NN$, let $\tau \coloneqq T/N$ denote the time step, set $u_0^\tau \coloneqq u_0$ and, for $k=1,\ldots , N$, let $u^\tau_k \in \dot{H}^2$ be a solution of the implicit Euler scheme
\eqnb\label{implicit_scheme}
\int \frac{u_k^\tau - u_{k-1}^\tau }{\tau } \psi = -\int \p_{xx} u_k^\tau  \psi_{xx} - \int \left( \p_x u_k^\tau \right)^2 \psi_{xx}
\eqne
for all $\psi \in \dot{H}^2$. The existence of such $u_k^\tau$ can be shown using the Lax--Milgram Lemma and the Leray--Schauder fixed point theorem.

For $t\in [(k-1)\tau , k\tau )$, $k\in \{1,\ldots , N\}$, let
\[
\begin{cases}
u^\tau (x,t) \coloneqq \frac{k\tau -t}{\tau } u_{k-1}^\tau (x) + \frac{t-(k-1)\tau }{\tau } u_k^\tau (x),\\
\overline{u}^\tau (x,t) \coloneqq u_k^\tau (x).
\end{cases}
\]
In other words $u^\tau$ denotes the linear approximation between the neighbouring $u_k^\tau$'s, and $\overline{u}^\tau$ denotes the next $u_k^\tau$.

Letting $\phi\coloneqq u_k^\tau$ in \eqref{implicit_scheme} and observing the cancellation
\[\int \left( \p_x u_k^\tau \right)^2 \p_{xx} u_k^\tau =0\] we obtain
\[
 \int \left( u_k^\tau \right)^2 + \tau \int \left( \p_{xx} u_k^\tau \right)^2 = \int \left( u_{k-1}^\tau \right)^2 , \qquad k\geq 1,
\]
from which, by summing in $k$, we obtain the energy inequality for $\overline{u}^\tau$,
\eqnb\label{EI_overline_u}
\| \overline{u}^\tau (t) \|^2 + \int_0^t \| \p_{xx} \overline{u}^\tau (s) \|^2 \d s \leq  \| u_0 \|^2,\qquad t\in (0,T) ,
\eqne
and similarly for $u^\tau$,
\eqnb\label{EI_u}
\| {u}^\tau (t) \|^2 + \int_0^t \| \p_{xx} {u}^\tau (s) \|^2 \d s \leq  C \| u_0 \|^2,\qquad t\in (0,T) .
\eqne
Furthermore, observe that for every $\psi\in\dot H^2$ and every $t\in[0,T)$ we have
\[
\int \p_t u^\tau (t) \psi = \int \frac{u_k^\tau - u_{k-1}^\tau }{\tau } \psi  = - \int  \left( \p_{xx} \overline{u}^\tau (t) + (\p_x \overline{u}^\tau (t))^2 \right)\psi_{xx},
\]
where $k\geq 1$ is such that $t\in [(k-1)\tau, k\tau )$. Thus, since each $u_k^\tau$, $k\geq 0$, has zero mean, the above equality holds in fact for all $\phi \in H^2$, that is
\eqnb\label{prepre_weak_form}
\int \p_t u^\tau (t) \psi  = - \int  \left( \p_{xx} \overline{u}^\tau (t) + (\p_x \overline{u}^\tau (t))^2 \right)\psi_{xx} , \qquad \psi \in {H}^2, t\in [0,T).
\eqne
Taking $\psi \coloneqq \phi (t)$ for some $\phi \in C_0^\infty (\TT \times [0,T))$ and integrating in time gives
\eqnb\label{pre_weak_form}
\int_0^T \int \p_t u^\tau  \phi  = - \int_0^T \int  \left( \p_{xx} \overline{u}^\tau  + (\p_x \overline{u}^\tau (t))^2 \right) \phi_{xx} , \qquad \phi \in C_0^\infty (\TT \times [0,T)).
\eqne
From here one can apply H\"older's inequality, the Sobolev embedding $H^{1/5} \subset L^{10/3}$, the Sobolev interpolation \eqref{Sobinterp}, \eqref{EI_overline_u} and a standard density argument to obtain a uniform (in $\tau$) estimate on $\p_t u^\tau$ in $L^{5/3} ((0,T);(W^{2,5/2})^*)$. This, the energy inequalities \eqref{EI_overline_u}, \eqref{EI_u}, and the Aubin--Lions lemma (see Theorem 2.1 in Section 3.2 in Temam \cite{temam}, for example) give the existence of a sequence $\tau_n\to 0^+$ and a $u\in L^2((0,T);W^{1,\infty })$  such that
\eqnb\label{conv1}
\begin{split}
u^{\tau_n} \to \,&u \qquad \text{ in } L^2((0,T);W^{1,\infty }),\\
 \overline{u}^{\tau_n}, u^{\tau_n} \rightharpoonup \,&u\qquad \text{ in } L^2((0,T);H^2),\\
  \overline{u}^{\tau_n}, u^{\tau_n} \stackrel{*}{\rightharpoonup} \,&u\qquad \text{ in } L^\infty ((0,T);L^2)
\end{split}
\eqne
as $\tau_n \to 0$. Here ``$\rightharpoonup$'' and ``$\stackrel{*}{\rightharpoonup}$'' denote the weak and weak-$*$ convergence, respectively. The fact that both $\overline{u}^{\tau_n}$ and $u^{\tau_n}$ converge to the same limit function follows from the convergence
\[
\| u^{\tau_n}  - \overline{u}^{\tau_n} \|_{L^2((0,T);W^{1,\infty })} \to 0 \quad \text{ as }\tau_n \to 0,
\]
which can be shown using the first convergence from \eqref{conv1}; see Lemma 2.3 in King et al. \cite{king_stein_winkler} for details.

The limit function $u$ is a weak solution to the surface growth initial value problem since the regularity requirement \eqref{weak_sol_regularity} follows from the convergence above and \eqref{weak_sol_distr_eq} follows by taking the limit $\tau_n \to 0^+$ in \eqref{pre_weak_form} after integration by parts in time of the left-hand side.\end{proof}
As with the partial regularity theory for the Navier--Stokes equations, we make key use of a local energy inequality. This gives rise to the notion of ``suitable weak solutions'', which we now define.

\begin{definition}[Suitable weak solution]\label{def_of_suitable_weak_sol}
We say that a weak solution is \emph{suitable} if the local energy inequality
\eqnb\label{LEI_alt_form}\begin{split}
\frac{1}{2} \int u(t)^2 \phi(t) + \int_0^t \int u_{xx}^2\phi &\leq \int_0^t \int \left(\frac{1}{2}(\phi_t-\phi_{xxxx})u^2 \right.\\
&\hspace{2cm}\left.+2u_x^2\phi_{xx}-\frac{5}{3}u_x^3\phi_x-u_x^2u\phi_{xx} \right)
\end{split}
\eqne
holds for all $\phi \in C_0^\infty (\TT \times (0,\infty ) ; [0,\infty ))$ and almost all $t\geq 0$.
\end{definition}

Note that the local energy inequality is a weak form of the inequality
\[
u \left( u_t + u_{xxxx} + \p_{xx} u_x^2 \right) \leq 0;
\]
that is \eqref{LEI_alt_form} can be obtained (formally) by multiplying the above inequality by $\phi$ and integrating by parts. We note that \eqref{LEI_alt_form} remains true if $u$ is replaced by $u-K$ for any $K\in \RR$. Indeed, 
multiplying  \eqref{weak_sol_distr_eq_alt} with $s\coloneqq 0$ by $K$ (and integrating by parts the term with four $x$ derivatives) we obtain
\[
-K \int u(t) \phi (t) = -K \int_0^t \int \left( u \phi_t - u \phi_{xxxx} - u_x^2 \phi_{xx} \right) .
\]
Thus noting that  
\[
\frac{K^2 }{2} \int \phi(t) = \int_0^t\int \frac{1}{2} \left( \phi_t - \phi_{xxxx} \right) K^2 
\]
(since $\phi$ has compact support in $\TT\times (0,\infty )$) we obtain the claim by adding the above two equalities from \eqref{LEI_alt_form}.

By adapting the method outlined above in the proof of the existence of a weak solution, we now show that this solution also satisfies the local energy inequality and is therefore `suitable'.

\begin{theorem}\label{thm:suitable}
The weak solution given by Theorem \ref{thm_existence_of_weak_sols} is suitable.
\end{theorem}

\begin{proof} Fix $\phi \in C_0^\infty (\TT \times (0,T) )$ with $\phi\geq 0$. We will show that 
\eqnb\label{LEI_to_show} \int_0^T \int u_{xx}^2\phi \leq \int_0^T \int \left(\frac{1}{2}(\phi_t-\phi_{xxxx})u^2 +2u_x^2\phi_{xx}-\frac{5}{3}u_x^3\phi_x-u_x^2u\phi_{xx} \right).
\eqne
This is equivalent to \eqref{LEI_alt_form}, which can be shown using a cut-off procedure (in time), similarly to the equivalence between \eqref{weak_sol_distr_eq} and \eqref{weak_sol_distr_eq_alt}.\\

Let $n$ be large enough so that $\phi(t)\equiv 0$ for $t\in (0,2\tau_n ) \cup (T-2\tau_n ,T)$. For brevity we will write $\tau$ in place of $\tau_n$. Given $t\in [0,T)$ set $\varphi \coloneqq \phi(t)$ and let $k$ be such that $t\in [(k-1)\tau , k\tau )$. Let $\psi \coloneqq u^\tau_k \varphi $ in \eqref{prepre_weak_form} to obtain
\eqnb\label{preLEI}
\int \frac{u_k^\tau - u_{k-1}^\tau }{\tau } u_k^\tau \varphi = -\int \p_{xx} u_k^\tau  \left( u_k^{\tau} \varphi \right)_{xx} - \int \left( \p_x u_k^\tau \right)^2 \left( u_k^{\tau} \varphi \right)_{xx} .
\eqne
Since integration by parts gives for any $v\in H^2$
\begin{align*}
\int v_{xx} v_x \varphi_x &= -\frac{1}{2} \int v_x^2 \varphi_{xx},\\
\int v_{xx} v \varphi_{xx} &= -\int v_x^2 \varphi_{xx} - \int v_x v \varphi_{xxx} = -\int v_x^2 \varphi_{xx} +\frac{1}{2} \int v^2  \varphi_{xxxx},\\
\int v_x^2 v_{xx} \varphi &= -\frac{1}{3} \int v_x^3 \varphi_x,
\end{align*}
the first term on the right-hand side of \eqref{preLEI} can be written in the form
\begin{equation*}
\begin{split}
-\int \p_{xx} u_k^\tau  \left( u_k^{\tau} \varphi \right)_{xx}& = -\int (\p_{xx} u_k^\tau )^2 \phi -2 \int \p_{xx} u_k^\tau \p_x u_k^\tau \varphi_x - \int \p_{xx} u_k^\tau \, u_k^\tau  \varphi_{xx}\\
&= -\int (\p_{xx} u_k^\tau )^2 \phi +2 \int (\p_{x} u_k^\tau)^2 \varphi_{xx} - \frac{1}{2} \int ( u_k^\tau )^2   \varphi_{xxxx}.
\end{split}
\end{equation*}
Similarly, the second term in \eqref{preLEI} can be expanded into
\begin{equation*}
\begin{split}
- \int \left( \p_x u_k^\tau \right)^2 \left( u_k^{\tau} \varphi \right)_{xx} &= - \int \left(\p_x u_k^\tau  \right)^2 \p_{xx} u_k^\tau  \varphi - 2 \int \left(\p_x u_k^\tau  \right)^2 \p_x u_k^\tau \varphi_x\\
&\qquad\qquad\qquad - \int \left(\p_x u_k^\tau  \right)^2 u_k^\tau \varphi_{xx} \\
&= -\frac{5}{3} \int \left(\p_x u_k^\tau  \right)^3  \varphi_x - \int \left(\p_x u_k^\tau  \right)^2 u_k^\tau \varphi_{xx} .
\end{split}
\end{equation*}
On the other hand, using the inequality $ab\leq a^2/2+b^2/2$ we can bound the left-hand side of \eqref{preLEI} from below by writing
\begin{equation*}\begin{split}
\int \frac{u_k^\tau - u_{k-1}^\tau }{\tau } u_k^\tau \varphi &= \frac{1}{\tau } \|   u_k^\tau \sqrt{ \varphi } \|^2 - \frac{1}{\tau } \int u_k^\tau \sqrt{ \varphi } \, u_{k-1}^\tau \sqrt{ \varphi }\\
&\geq \frac{1}{2\tau } \|   u_k^\tau \sqrt{ \varphi } \|^2 - \frac{1}{2\tau }  \| u_{k-1}^\tau \sqrt{ \varphi } \|^2.
\end{split}
\end{equation*}
Substituting these calculations into \eqref{preLEI} gives
\begin{equation*}\begin{split}
\frac{1}{2\tau } &\| u_k^\tau \sqrt{ \varphi } \|^2 -\frac{1}{2\tau } \| u_{k-1}^\tau  \sqrt{ \varphi }\|^2 + \int (\p_{xx} u_k^\tau )^2 \varphi \\
&\leq  \int \left( 2 (\p_{x} u_k^\tau)^2 \varphi_{xx} - \frac{1}{2} ( u_k^\tau )^2 \varphi_{xxxx} -\frac{5}{3}  \left(\p_x u_k^\tau  \right)^3  \varphi_x -  \left(\p_x u_k^\tau  \right)^2 u_k^\tau \varphi_{xx} \right).
\end{split}
\end{equation*}
Integration in time gives
\eqnb\label{pre1_LEI}\begin{split}
\frac{1}{2\tau } &\int_0^T \| \overline{u}^\tau (t) \sqrt{ \phi (t) } \|^2 \d t -\frac{1}{2\tau } \int_0^T \| \overline{u}^\tau (t-\tau ) \sqrt{ \phi (t) } \|^2\d t + \int_0^T\int (\overline{u}_{xx}^\tau )^2 \phi \\
&\leq \int_0^T \int \left( 2 (\overline{u}_x^\tau)^2 \phi_{xx} - \frac{1}{2} ( \overline{u}^\tau )^2 \phi_{xxxx} -\frac{5}{3}  \left(\overline{u}_x^\tau  \right)^3 \phi_x -  \left(\overline{u}_x^\tau  \right)^2 \overline{u}^\tau \phi_{xx} \right).
\end{split}
\eqne
Observe that the convergence $\overline{u}^\tau \to u$ in $L^2 ((0,T);W^{1,\infty } )$ (see \eqref{conv1}) gives the convergence of the right-hand side above to the respective expression with $u$,
\[
\int_0^T \int \left( 2 u_x^2 \phi_{xx} - \frac{1}{2} u^2 \phi_{xxxx} -\frac{5}{3} u_x^3  \phi_x -  u_x^2 u \phi_{xx} \right).
\]
Moreover, the weak convergence $\overline{u}^\tau \rightharpoonup u$ in $L^2 ((0,T);H^2)$ (see \eqref{conv1}) gives in particular the weak convergence
\[
\overline{u}^\tau_{xx} \sqrt{\phi } \rightharpoonup u_{xx} \sqrt{ \phi } \quad \text{ in } L^2((0,T);L^2 ) \text{ as } \tau \to 0,
\]
and thus, from properties of weak limits,
\[
\int_0^T \int u_{xx}^2 \phi \leq \liminf_{\tau \to 0 } \int_0^T \int \left( \overline{u}^\tau_{xx} \right)^2 \phi .
\]
As for the first two terms in \eqref{pre1_LEI}, they can be written in the form
\eqnb\label{temp1_first_two_terms}\begin{split}
\int_0^T  &\frac{\| \overline{u}^\tau (t) \sqrt{\phi (t) } \|^2 - \| \overline{u}^\tau (t-\tau ) \sqrt{\phi (t-\tau ) } \|^2 }{2\tau } \d t \\
&\hspace{2cm}- \frac{1}{2 } \int_0^T \left( ( \overline{u}^\tau (t-\tau) )^2 , \frac{\phi (t) - \phi (t-\tau )}{\tau } \right)\d t,
\end{split}
\eqne
where $(\cdot, \cdot )$ denotes the $L^2$ product. Observe that the first term vanishes due to the change of variable $t'\coloneqq t-\tau$ and the fact that $\phi $ vanishes on time intervals $(0,2\tau)$ and $(T-2\tau , T)$. A similar change of variables in the second term gives that
\eqref{temp1_first_two_terms} equals
\[
-\frac{1}{2} \int_0^T \left( ( \overline{u}^\tau (t) )^2 , \frac{\phi (t+\tau ) - \phi (t )}{\tau } \right)\d t .
\]
Thus the convergence $\overline{u}^\tau \to u$ in $L^2 ((0,T);W^{1,\infty } )$ and the fact that
\[
\frac{\phi (x,t+\tau ) - \phi (x,t )}{\tau } \to \phi_t (x,t) \qquad \text{ uniformly in } (x,t)\in \TT \times (0,T)
\]
give that \eqref{temp1_first_two_terms} converges to
\[
 -\frac{1}{2} \int_0^T \int u^2 \phi_t
\]
as $\tau \to 0^+$. Hence, altogether, taking $\liminf_{\tau \to 0^+}$ (recall we write $\tau$ in place of $\tau_n$) in \eqref{pre1_LEI} gives the local energy inequality \eqref{LEI_to_show}, as required. 
\end{proof}

\section{A `nonlinear' parabolic Poincar\'e inequality}\label{sec_PPI}

Here we prove a parabolic version of the Poincar\'e inequality, which is a key ingredient in the proof of the partial regularity results that follow.

\begin{theorem}[Parabolic Poincar\'e inequality]\label{PPI}
Let $\eta \in [0,1]$, $r\in (0,1)$ and let $Q(z_0,r)$ be a cylinder, where $z_0=(x_0,t_0)$. If a function $u$ satisfies
\eqnb\label{distr_form_for_poincare}
\int_{B_{r} (x_0)} (u(t)-u(s)) \phi  = \int_s^t \int_{B_{r}(x_0)}  u_x \phi_{xxx} -\eta  \int_s^t \int_{B_{r}(x_0)} u_x^2 \phi_{xx}
\eqne
for all $\phi \in C_0^\infty ( B_{r}(x_0) )$ and almost every $s,t\in (-r^4,r^4)$ with $s<t$, then
\eqnb\label{PPoinc}
\frac{1}{r^5}\int_{Q(z_0,r/2)}|u-u_{z_0,r/2}|^3\le c_{pp} \left( \Y (z_0,r) +\eta \Y (z_0,r)^2 \right),
\eqne
where 
\eqnb\label{def_of_Y}
\Y (z_0,r) \coloneqq \frac{1}{r^2} \int_{Q(z_0,r)} |u_x|^3
\eqne
and $c_{pp}>0$ is an absolute constant.
\end{theorem}
Recall $u_{z_0,r/2}$ denotes the mean of $u$ over $Q(z_0,r/2)$ (see \eqref{notation_mean_over_cylinder}). Note that no $t$ derivative appears on the right-hand side of \eqref{PPoinc}.
Observe that \eqref{PPoinc} is the classical Poincar\'e inequality if $\eta =0$ and the left-hand side is replaced by
$$
\frac{1}{r^5}\int_{Q(z_0,r/2)}\left|u-\fint_{B(x_0,r/2)} u(t)\right|\,\d x\,\d t
$$
(i.e.\ the mean over the cylinder is replaced by the mean over the ball at each time).
Moreover note that \eqref{PPoinc} does not hold for arbitrary functions since adding a function of time to $u$ allows one to increase the left-hand side while keeping the right-hand side bounded. This also verifies the relevance of the assumption \eqref{distr_form_for_poincare} since it shows that the only function of time which can be added to $u$ is a constant function. On the other hand, adding constants to $u$ makes no change to \eqref{PPoinc}.

Furthermore, the case $\eta =0$ gives the parabolic Poincar\'e inequality for weak solutions to the biharmonic heat equation:
$$
\frac{1}{r^3}\int_{Q(z_0,r/2)}|u-u_{z_0,r/2}|^3\le c_{pp}\int_{Q(z_0,r)}|u_x|^3,
$$
whenever $u_t=\p_x^4 u$ (weakly). In this case it can be shown that the inequality holds in any dimension (with $\p_x^4$ replaced by the bilaplacian $\Delta^2$) and for any $p\geq 1$.

Due to \eqref{weak_sol_distr_eq_alt} any weak solution of the surface growth equation satisfies \eqref{distr_form_for_poincare} for all $z_0$, $r$ as long as $Q(z_0,r)\subset \TT \times (0,T)$, and hence we can use inequality \eqref{PPoinc} for the suitable weak solutions that form our main subject in what follows.

We prove this nonlinear parabolic Poincar\'e inequality adapting the approach used by Aramaki \cite{aramaki} in the context of the heat equation, itself based on previous work by Struwe \cite{struwe}.

\begin{proof} Fix $r$ and $z_0$ and set, for brevity
\[
Q_\rho \coloneqq Q (z_0,\rho ),\quad B_\rho \coloneqq B (x_0, \rho)\qquad \text{for }\rho >0, \text{ where } z_0 = (x_0,t_0),
\]
and set
\[
M\coloneqq \Y(z_0,r).
\]

\noindent\emph{Step 1.} We introduce the notion of $\sigma$-means.\\

Let $\sigma \colon \RR \to [0,1]$ be the cut-off function in space around $x_0$ such that
\[
\sigma (x) =\begin{cases}
1\quad &|x-x_0|\leq r/2 ,\\
0 & |x-x_0| \geq r,
\end{cases}\qquad |\p_x^k \sigma | \leq C r^{-k}, \,\, k\geq 0.
\]
Let
\eqnb\label{sigma_means_notation}
u^\sigma_r(t) \coloneqq \frac{\int_{B_{r}} u(t) \sigma\,\d x }{\int_{B_{r}} \sigma\,\d x },\qquad [u]^\sigma_{r} \coloneqq \frac{\int_{Q_{r}} u \sigma\,\d z }{\int_{Q_{r}} \sigma\,\d z }
\eqne
denote the \emph{$\sigma$-mean} of $u$ over a ball (at a given time $t$) and over a cylinder, respectively. Note that, since $\sigma$ is a function of $x$ only,
\eqnb\label{fact3_eq}
u^\sigma_{r}(t) - [u]^\sigma_r = \frac{1}{2r^4} \int_{-r^4}^{r^4 } \left( u^\sigma_{r}(t) - u^\sigma_{r}(s) \right) \d s.
\eqne
Furthermore, let us write for brevity
\[
u_{r} \coloneqq u_{z_0,{r}};
\]
then
 \eqnb\label{fact2_eq}
 \int_{Q_{r/2}} |u - u_{r/2} |^3 \leq 8 \int_{Q_{r/2}} |u- [u]_r^\sigma |^3.
 \eqne
Indeed, by writing
 \[
 |u_{r/2} - L |^3 = \left| \frac{1}{|Q_{r/2}|} \int_{Q_{r/2}} u-L \right|^3 \leq \frac{1}{|Q_{r/2}|} \int_{Q_{r/2}} \left| u-L \right|^3 ,
 \]
where $L\coloneqq [u]_r^\sigma$, we see that the triangle inequality gives
\begin{equation*}\begin{split}
\left( \int_{Q_{r/2}} |u - u_{r/2} |^3 \right)^{1/3} &\leq  \left( \int_{Q_{r/2}} |u- L |^3 \right)^{1/3} +\left( \int_{Q_{r/2}} |u_{r/2}- L |^3 \right)^{1/3}  \\
&\leq 2 \left( \int_{Q_{r/2}} |u- L |^3 \right)^{1/3},
\end{split}
\end{equation*}
as required. In what follows we will also use the following classical Poincar\'e inequality: for $t\in (0,T)$, $q\geq 1$, $r\in (0,1)$,
\eqnb\label{fact1_eq}
\int_{B_{r }} | u(t) - u^\sigma_r(t) |^q \sigma  \leq C(n,q) r^q \int_{B_{r}} |u_x (t)|^q \sigma ,
\eqne
see Lemma 6.12 in Lieberman \cite{lieberman} for a proof.\vspace{0.4cm}\\
\emph{Step 2.} We show that for almost every $s,t \in (-r^4,r^4)$
\eqnb\label{estimate_means_s_t}
|u^\sigma_r(t) - u^\sigma_r(s) |^{3}\le C (M+\eta M^2).
\eqne

To this end suppose (without loss of generality) that $s<t$ and let
\[
\phi (x) \coloneqq \sigma (x) (u^\sigma_r(t) - u^\sigma_r(s) ) |u^\sigma_r(t) - u^\sigma_r(s) |,
\]
be the test function in \eqref{distr_form_for_poincare}. Then the term on the left-hand side can be bounded from below,
\begin{align*}
\int_{B_{r}} (u(t) - u(s) ) \phi &= (u^\sigma_r(t) - u^\sigma_r(s) ) |u^\sigma_r(t) - u^\sigma_r(s) |  \int_{B_{r}} (u(t) - u(s) ) \sigma \\
&= |u^\sigma_r(t) - u^\sigma_r(s) |^{3}\int_{B_{r}} \sigma \geq C r|u^\sigma_r(t) - u^\sigma_r(s) |^{3}.
\end{align*}
The first term on the right-hand side can be estimated by writing
\begin{equation*}
\begin{split}
\left| \int_s^t \int_{B_{r}} u_x \phi_{xxx} \right| &\leq |u^\sigma_r(t) - u^\sigma_r(s) |^{2}  \int_s^t \int_{B_{r}} |u_x| \, |\sigma_{xxx} |\\
&\leq C |u^\sigma_r(t) - u^\sigma_r(s) |^{2} r^{-3}  \int_{Q_{r}} |u_x|\\
&\leq C |u^\sigma_r(t) - u^\sigma_r(s) |^{2} r^{-3}  \left( \int_{Q_{r}} |u_x|^3 \right)^{1/3} r^{10/3}\\
&\leq \delta r |u^\sigma_r(t) - u^\sigma_r(s) |^{3} + C_\delta r^{-1} \int_{Q_{r}} |u_x|^3\\
&= \delta r |u^\sigma_r(t) - u^\sigma_r(s) |^{3} + rMC_\delta
\end{split}
\end{equation*}
for any $\delta >0$, where we used H\"older's inequality and Young's inequality in the form
$$a^2 r^{1/3} b^{1/3} \leq \delta a^3 r + C_\delta b r^{-1}.$$
The second term on the right-hand side can be estimated by writing
\begin{equation*}
\begin{split}
\left| \int_s^t \int_{B_{r}} u^2_x \phi_{xx} \right| &\leq |u_r^\sigma(t) - u_r^\sigma(s) |^{2} \int_s^t \int_{B_{r}} |u_x|^2 \, |\sigma_{xx} |\\
&\leq Cr^{-2}|u_r^\sigma(t) - u_r^\sigma(s) |^{2}  \int_{Q_{r}} |u_x|^2 \\
& \leq Cr^{-2}|u_r^\sigma(t) - u_r^\sigma(s) |^{2}  \left( \int_{Q_{r}} |u_x|^3 \right)^{2/3} r^{5/3} \\
&\leq \delta r |u_r^\sigma(t) - u_r^\sigma(s) |^{3} + C_\delta r^{-3} \left( \int_{Q_{r}} |u_x|^3 \right)^{2}\\
&=\delta r |u_r^\sigma(t) - u_r^\sigma(s) |^{3} + C_\delta r M^2 , \qquad \delta >0,
\end{split}
\end{equation*}
where we used H\"older's inequality and Young's inequality in the form
$$
a^2r^{-1/3} b^{2/3} \leq \delta r a^3 + C_\delta r^{-3} b^2.
$$
Since $\eta \leq 1$ (see \eqref{distr_form_for_poincare}) we therefore obtain
$$
Cr|u^\sigma_r(t) - u^\sigma_r(s) |^{3}\le 2\delta r|u_r^\sigma(t) - u_r^\sigma(s) |^{3} + rC_\delta(M+\eta M^2),
$$
and fixing $\delta>0$ sufficiently small gives \eqref{estimate_means_s_t}.\vspace{0.4cm}\\
\emph{Step 3.} We show \eqref{PPoinc}.\\

From \eqref{fact2_eq}, the fact that $\sigma \in [0,1]$ with $\sigma =1$ on $Q_{r/2}$ and the inequality $\int |f+g|^q\leq 2^q \int |f|^q + 2^q \int |g|^q$ we obtain
\eqnb\label{estimate_for_poincare}
\begin{split}
\int_{Q_{r/2}} |u-u_{r/2} |^3 &\leq 8 \int_{Q_{r}} | u - [u]_{r}^\sigma |^3 \sigma \, \d x \, \d t \\
&\leq 64 \int_{Q_{r}} | u - u_r^\sigma (t)|^3 \sigma \, \d x \, \d t + 64 \int_{Q_{r}} | u_r^\sigma (t)- [u]_{r}^\sigma |^3\d x \, \d t .
\end{split}
\eqne
The first of the resulting integrals can be bounded using \eqref{fact1_eq},
\[
\int_{Q_{r}} | u - u_r^\sigma (t) |^3 \sigma \,\d x \, \d t \leq Cr^3 \int_{Q_{r}} |u_x |^3 \sigma \leq  Cr^5 M .
\]
The second one can be bounded using \eqref{fact3_eq} and Step 2,
\eqnb\label{est_of_the_means_other}
\left| u^\sigma_{r}(t) - [u]^\sigma_r \right|^3 \leq  \frac{1}{2r^4} \int_{-r^4}^{r^4 } \left| u^\sigma_{r}(t) - u^\sigma_{r}(s) \right|^3 \d s \leq C (M+\eta M^2),
\eqne
which gives
\[
\int_{Q_{r}} | u_r^\sigma (t) - [u]_{r}^\sigma |^3\d x \, \d t \leq C r^5  (M+\eta M^2)
\]
Applying these bounds in \eqref{estimate_for_poincare} gives
\[
\int_{Q_{r/2}} |u-u_{r/2} |^3 \leq C r^5 (M+\eta M^2),
\]
that is \eqref{PPoinc}.
\end{proof}
\begin{corollary}\label{cor_to_poincare}
The claim of Theorem \ref{PPI} remains valid if \eqref{PPoinc} is replaced by 
\[
\frac{1}{r^5}\int_{Q(z_0,r)}\left| u-[u]^\sigma_{r} \right|^3\sigma \,\le c \left( \Y (z_0,r) +\eta \Y (z_0,r)^2 \right),
\]
\end{corollary}
\begin{proof}
This follows by ignoring the first inequality in \eqref{estimate_for_poincare}.
\end{proof}
\section{The first conditional and partial regularity results}\label{sec_1st_local_reg}

Here we show local regularity of suitable weak solutions to the surface growth equation based on a condition on $u_x$. Namely, we will show in Theorem \ref{localreg1} that there exists $\varepsilon_0 >0$ and $R_0>0$ such that if
\[
\frac{1}{r^2} \int_{Q(z,r)} |u_x |^3 < \varepsilon_0
\]
for some $r<R_0$ and $z$ then $u$ is H\"older continuous in $Q(z,r/2)$.

The proof we give of this result is based on that presented for the Navier--Stokes equations by Ladyzhenskaya \& Seregin \cite{ladyzhenskaya_seregin}; we begin with a certain `one-step' decay estimate, which we then iterate.

\subsection{Interior regularity for the biharmonic heat flow}

The proof of the decay estimate relies on the following regularity result for the biharmonic heat equation; while the result is perhaps `standard', we could not find an obvious canonical reference, and so for the sake of completeness we provide a short proof.

\begin{proposition}[Interior regularity of the biharmonic heat flow]\label{prop_magic}
Suppose that $0<b<a$, $v,v_x\in L^2(Q_a)$ and that $v$ is a distributional solution to the biharmonic heat equation $v_t=-v_{xxxx}$ in $Q_a$, that is
\begin{equation}\label{limiteq}
\iint_{Q_a} v \, \phi_t = \iint_{Q_a} v\, \phi_{xxxx}
\end{equation}
for every $\phi\in C_0^\infty(Q_a)$. Then
$$
\|v_x\|_{L^\infty(Q_{b})}\le C_{a,b}\left(\|v\|_{L^2(Q_a)}+\|v_x\|_{L^2(Q_a)}\right)
$$
for some $C_{a,b}>0$.
\end{proposition}
\begin{proof} We assume that $a=1$, $b=1/2$; the claim for arbitrary $a$, $b$ follows similarly.
First we show that $v_{xx}\in L^2 (Q_{7/8} )$ with
\eqnb\label{u_xx_is_in_L^2}
\|  v_{xx} \|_{L^2(Q_{7/8}  )} \leq C \left(\|v\|_{L^2(Q_1)}+\|v_x\|_{L^2(Q_1)}\right).
\eqne

For this let $\varepsilon \in (0,1/16)$. Then  $ \phi^{(\varepsilon )}\in C_0^\infty(Q_1)$ for every $\phi \in C_0^\infty (Q_{15/16})$, where $\phi^{(\varepsilon )}$ denotes the standard mollification (in both space and time) of $\phi$. Using $\phi^{(\varepsilon )}$ as a test function in \eqref{limiteq} and applying the Fubini Theorem we obtain
\[
\iint_{Q_{15/16}} v^{(\varepsilon )} \,\phi_t = \iint_{Q_{15/16}} v^{(\varepsilon )} \,\phi_{xxxx}, \qquad \phi \in C_0^\infty (Q_{15/16}),
\]
that is $v^{(\varepsilon )}$ is a distributional solution of the biharmonic heat equation in $Q_{15/16}$. Moreover, from properties of mollification,
\eqnb\label{mollification_is_bdd}
\|v^{(\varepsilon )}\|_{L^2(Q_{15/16})} \leq \| v\|_{L^2(Q_{1})} \qquad\mbox{and}\qquad  \|v^{(\varepsilon )}_x \|_{L^2(Q_{15/16})} \leq \| v_x \|_{L^2(Q_{1})} \eqne
for all $\varepsilon $. Since $v^{(\varepsilon )}$ is smooth it satisfies the equation
$$
v^{(\varepsilon )}_t=-v^{(\varepsilon )}_{xxxx}
$$
in the classical sense. Multiplying this equation by $v^{(\varepsilon )} \phi$ (where $\phi \in C_0^\infty (Q_{15/16})$) and integrating by parts on $Q_{15/16}$ gives
\[\iint_{Q_{15/16}} ( v^{(\varepsilon )}_{xx})^2\phi=\iint_{Q_{15/16}}\left( \frac{1}{2}( v^{(\varepsilon )} )^2(\phi_t-\phi_{xxxx} ) +2( v^{(\varepsilon )}_x )^2\phi_{xx}\right)
\]
for every $\phi \in C_0^\infty (Q_{15/16} )$. Taking $\phi \geq 0$ such that $\phi=1$ on $Q_{7/8}$ we obtain
\[
\|  v^{(\varepsilon )}_{xx} \|_{L^2 (Q_{7/8})} \leq C_\rho  \left(\|v\|_{L^2(Q_1)}+\|v_x\|_{L^2(Q_1)}\right),
\]
where we used \eqref{mollification_is_bdd}. Thus $v^{(\varepsilon )}_{xx}$ is a bounded in $L^2 (Q_{7/8})$ and hence there exists a sequence $\varepsilon_k \to 0^+$ such that $v^{(\varepsilon )}_{xx} \rightharpoonup v_{xx}$ weakly in $L^2 (Q_{7/8})$. Note that the limit function is $v_{xx}$ by definition of weak derivatives since $v^{(\varepsilon )} \to v$ strongly in $L^2 (Q_{15/16})$. Thus in particular $v_{xx}\in L^2 (Q_{7/8 })$ and, using a property of weak limits and the last inequality, we obtain
\[
\| v_{xx} \|_{L^2 (Q_{7/8})} \leq \liminf_{\varepsilon_k\to 0^+} \|  v^{(\varepsilon_k )}_{xx} \|_{L^2 (Q_{7/8})} \leq C_\rho  \left(\|v\|_{L^2(Q_1)}+\|v_x\|_{L^2(Q_1)}\right),
\]
that is \eqref{u_xx_is_in_L^2}, as required.

Now letting $\phi\coloneqq \psi_x$ for some $\psi \in C_0^\infty ( Q_{7/8} )$ we see from \eqref{limiteq} that $v_x$ is a distributional solution of the biharmonic heat equation in $Q_{7/8}$. Moreover, using \eqref{u_xx_is_in_L^2}, we see that $v_x,v_{xx}\in L^2 (Q_{7/8})$. Thus applying a similar argument as in the case of \eqref{u_xx_is_in_L^2} we obtain that $v_{xxx} \in L^2( Q_{3/4} )$ with
\[
\| v_{xxx} \|_{L^2 (Q_{3/4})} \leq C \left(\|v\|_{L^2(Q_1)}+\|v_x\|_{L^2(Q_1)}\right) .
\]
In the same way we observe that any spatial derivative of $v$ is a distributional solution of the biharmonic heat equation, and $\p^k_x v\in L^2 (Q_{1/2})$ for all $k\leq 9$ with
\[
\| \p^k_{x}v  \|_{L^2 (Q_{1/2})} \leq C \left(\|v\|_{L^2(Q_1)}+\|v_x\|_{L^2(Q_1)}\right) , \qquad k\leq 9.
\]
Now since \eqref{limiteq} gives in particular that $v_t=-v_{xxxx}$ in the sense of weak derivatives, we obtain from the above that each of $v_x,v_{xx},v_{xt},v_{xxx},v_{xxt},v_{xtt}$ is bounded in $L^2 (Q_{1/2})$ by $ C \left(\|v\|_{L^2(Q_1)}+\|v_x\|_{L^2(Q_1)}\right)$. Therefore the claim of the lemma follows from the two-dimensional embedding $H^2 \subset L^\infty $.
\end{proof}
%

\subsection{The `one-step' estimate}
Let $u$ be a suitable weak solution of the surface growth model. In what follows we assume that a cylinder $Q (z,r)$ is contained in $\TT\times (0,\infty )$, the domain of definition of $u$.
We now state and prove the `one-step' estimate.

\begin{lemma}\label{onedecay}
  Given $\theta\in(0,1/4)$ there exist $\varepsilon_*= \varepsilon_* (\theta )$ and  $R= R(\theta )$ such that if$r<R$ and
$$
  \Y(z,r):=\frac{1}{r^2}\int_{Q(z,r)}|u_x|^3 <\varepsilon_*
$$
  then
\be{Ydecay1}
  \Y(z,\theta r)\le c_* \theta^{3} \Y(z,r),
\ee
where $c_*$ is a universal constant.
\end{lemma}
\begin{proof} We will show the claim for
\[
c_* \coloneqq 8 {C}_{1/2,1/4}^3 \left( 1+ c_{pp}^{1/3}\right)^3,
\]
where $C_{1/2,1/4}$ is the constant from Proposition \ref{prop_magic} and $c_{pp}$ is from the parabolic Poincar\'e inequality (Theorem \ref{PPI}).
  Suppose that the claim is not true. Then there exist $r_k\to0$, $\varepsilon_k\to0$, and $z_k=(x_k,t_k)$ such that
$$
 \Y(z_k,r_k)=\frac{1}{r_k^2}\int_{Q(z_k,r_k)}|u_x|^3=\varepsilon_k,
$$
  but
$$
\frac{1}{ (\theta r_k)^2}\int_{Q(z_k,\theta r_k)}|u_x|^3\ge c_* \theta^3\varepsilon_k.
$$
\emph{Step 1.} We take a limit of rescaled solutions.\\

Let
$$
   u_k(x,t)\coloneqq  \frac{u \left(x_k+ x \,r_k,t_k + t\,r_k^4\right) - u_{z_k,{r_k/2}} }{\varepsilon_k^{1/3} }
$$
be a family of rescalings of $u$. Then $\{u_k \}$ is a family of functions such that $\int_{Q_{1/2}} u_k =0$ (which will be used shortly when we apply the parabolic Poincar\'e inequality),
\begin{equation}\label{always1}
\int_{Q_1}|\partial_xu_k|^3=1,
\end{equation}
 \begin{equation}\label{lbd}
\int_{Q_{\theta} }|\partial_xu_k|^3\ge c_* \theta^{5},
 \end{equation}
 and $u_k$ satisfies the local energy inequality
\eqnb\label{u_k_LEI}\begin{split}
\int_{B_1} |u_k(t)|^2 \phi(t)&+\int_{-1}^t \int_{B_1} (\p_{xx} u_k)^2\phi \leq \int_{-1}^t \int_{B_1} \left(\frac{1}{2}(\phi_t-\phi_{xxxx})(u_k)^2 \right.\\
&\hspace{1cm}\left.+2(\p_x u_k)^2\phi_{xx}-\frac{5 }{3}\varepsilon_k^{1/3}(\p_x u_k)^3\phi_x-\varepsilon_k^{1/3}(\p_x u_k)^2u_k\phi_{xx} \right)
\end{split}
\eqne
for all nonnegative $\phi \in C_0^\infty (Q_1)$ and almost all $t\in (-1,1)$ (recall \eqref{LEI_alt_form}). Moreover $u_k$ satisfies the equation $ \partial_tu_k=-\partial_x^4u_k-\varepsilon_k^{1/3} \partial_{xx}(\partial_xu_k)^2$ in $Q_1$ in the sense of distributions, that is
\begin{equation}\label{BRrescaled}
\iint_{Q_1} u_k\,\phi_t = \iint_{Q_1} u_k\, \phi_{xxxx} + \varepsilon_k^{1/3}  \iint_{Q_1} \left( \p_x u_k \right)^2 \phi_{xx}, \quad \phi \in C_0^\infty (Q_1).
\end{equation}
It follows from the parabolic Poincar\'e inequality (Theorem \ref{PPI}) and \eqref{always1} that
\eqnb\label{l3_bound_on_u_k}
\int_{Q_{1/2}} |u_k |^3 \leq {c_{pp}} (1+\varepsilon_k^{1/3}  ) .
\eqne
Thus both $u_k$ and $\p_x u_k$ are bounded in $L^3 (Q_{1/2})$ and hence there exists $v\in L^3 (Q_{1/2})$ such that $\|v\|_{L^3 (Q_{1/2})} \leq c_{pp}^{1/3}$, $\|v_x\|_{L^3 (Q_{1/2})} \leq 1$ and
\[
u_{k_n} \rightharpoonup v,\quad \p_x u_{k_n} \rightharpoonup v_x \qquad \text{ in }L^3 (Q_{1/2})\text{ as } n\to \infty
\]
for some sequence $k_n\to \infty $. Taking the limit in \eqref{BRrescaled} we obtain
\[
\iint_{Q_{1/2}} v\,\phi_t = \iint_{Q_{1/2}} v\, \phi_{xxxx} , \quad \phi \in C_0^\infty (Q_{1/2}),
\]
that is the limit function $v$ is a distributional solution of the biharmonic heat equation $v_t=-v_{xxxx}$ on $Q_{1/2}$. In particular, using Proposition \ref{prop_magic}, we obtain
\eqnb\label{Linfty_bound_on_v}\begin{split}
\| v_x\|_{L^\infty(Q_{1/4})} &\leq C_{1/2,1/4} \left( \| v \|_{L^2 (Q_{1/2})} + \| v_x \|_{L^2 (Q_{1/2})}  \right) \\
&\leq C_{1/2,1/4} (1+c_{pp}^{1/3})=(c_*/8)^{1/3}.
\end{split}
\eqne
\emph{Step 2.} We show strong convergence $\p_x u_{k_n}\to v_x$ in $L^3(Q_{1/4})$ on a subsequence $k_n$ (relabelled).\\

We will write $k\coloneqq k_n$ for brevity. Letting $\phi\in C_0^\infty (Q_{1/2})$ be nonnegative and such that $\phi=1$ on $Q_{1/4}$ the local energy inequality \eqref{u_k_LEI} gives
\[
 \| u_k (t) \|^2_{L^2 (B_{1/4})} +\int_{-4^{-4}}^t  \| \p_{xx} u_k (s) \|^2_{L^2 (B_{1/4})}\d s \leq C
\]
for almost every $t\in (-4^{-4},4^{-4})=: I_{1/4}$,
where we also used \eqref{l3_bound_on_u_k}, \eqref{always1} and the fact that $\varepsilon_k <1$, and thus
\eqnb\label{EI_u_k}
\| u_k \|_{L^\infty (I_{1/4}; L^2(B_{1/4}))} + \| \p_{xx} u_k \|_{L^2(Q_{1/4})} \leq C.
\eqne
Using 1D Sobolev interpolation $ \|v\|_{H^{4/3}}\le\|v\|_{L^2}^{1/3}\|v\|_{H^2}^{2/3}$ (recall \eqref{Sobinterp}) this in particular gives
\eqnb\label{u_k_is_L3_int_H4/3}
\| u_k \|_{L^3 (I_{1/4};H^{4/3}(B_{1/4}))} \leq C.
\eqne
Moreover, from \eqref{BRrescaled} we obtain
\begin{equation*}\begin{split}
\left| \iint_{Q_{1/4}} \p_t u_k\,\phi \right| &= \left| -\iint_{Q_{1/4}} \p_{xx} u_k\, \phi_{xx} - \varepsilon_k^{1/3} \iint_{Q_{1/4}} \left( \p_x u_k \right)^2 \phi_{xx}\right|\\
&\leq \| \phi \|_{L^3 (I_{1/4};W^{2,3}(B_{1/4}))} \left( \| \p_{xx} u_k \|_{L^{3/2} (Q_{1/4})} +\| \p_{x} u_k \|^2_{L^{3} (Q_{1/4})}  \right)\\
&\leq C \| \phi \|_{L^3 (I_{1/4};W^{2,3}(B_{1/4}))}
\end{split}
\end{equation*}
for all $\phi \in C_0^\infty (Q_{1/4})$, where the last inequality follows from H\"older's inequality, the bound \eqref{EI_u_k} above and \eqref{always1}. By the density of $C_0^\infty (Q_{1/4})$ in $L^3 (I_{1/4};W^{2,3}(B_{1/4}))$ the above inequality gives boundedness of $\p_t u_k $ in ${L^{3/2} (I_{1/4};(W^{2,3}(B_{1/4}))^*)} $. This and \eqref{u_k_is_L3_int_H4/3} let us use the Aubin--Lions compactness lemma (see, for example, Section 3.2.2 in Temam, 2001\nocite{temam}) to extract a subsequence of $(u_k)$ (which we relabel) that converges in $L^3 (I_{1/4}; H^{7/6}(B_{1/4}))$. Using the 1D Sobolev embedding $H^{1/6} \subset L^3$ this in particular means that $\p_x u_k$ converges in $L^3 (Q_{1/4})$, as required.\\

\noindent\emph{Step 3.} We use \eqref{lbd} to obtain a contradiction.\\

Since $\theta \in (0,1/4)$ the last step gives in particular $\p_x u_{k_n} \to v_x$ in $L^3 (Q_{\theta })$. Thus taking the limit $k_n\to \infty $ in \eqref{lbd} and using the $L^\infty$ bound on $v_x$ from \eqref{Linfty_bound_on_v} we obtain
\[
1 \leq \frac{1}{c_* \theta^5}\int_{Q_{\theta} }|v_x|^3 \leq \frac{1}{8 \theta^5 }  |Q_{\theta }|=\frac{1}{2},
\]
a contradiction.\end{proof}

\subsection{Conditional regularity in terms of $u_x$}

We now iterate this estimate.

 \begin{lemma} Given $\alpha \in (0,3)$ there exist $\varepsilon_*>0$ and $R\in (0,1)$ such that if $r<R$ and
  \be{firststep}
  \frac{1}{r^2}\int_{Q(z,r)}|u_x|^3<\varepsilon_*
  \ee
  then
\eqnb\label{decayest}
\frac{1}{\varrho^2}\int_{Q(z,\varrho )}|u_x|^3 \le C  \eps_* \left( \frac{\varrho }{ r} \right)^{\alpha}
\qquad\mbox{for all}\quad \varrho\le r.
\eqne
\end{lemma}
\begin{proof}
Similarly as before we will use the notation $\Y(z,r)=\frac{1}{r^2}\int_{Q(z,r)}|u_x|^3$. Fix $\theta\in (0,1/2)$ sufficiently small such that
\[
c_* \theta^{3} < \theta^{\alpha}.
\]
Lemma \ref{onedecay} then guarantees that if $\Y(z,r)<\eps_*$ for some $r<R$ then
  $$
  \Y(z,\theta r)\le \theta^\alpha \,\Y(z,r).
  $$
  Iterating this result we obtain
$$
  \Y(z,\theta^k r)\le\theta^{ \alpha k }\Y(z,r), \qquad k\geq 0.
$$
  Now for $\varrho\in(0,r)$ choose $k$ such that
$$
  \theta^{k+1}r<\varrho\le\theta^kr;
$$
then
\begin{align*}
\Y(z,\varrho)&=\frac{1}{\varrho^2}\int_{Q(z,\varrho)}|u_x|^3\\
&\le\frac{1}{(\theta^{(k+1)}r)^2}\int_{Q(z,\theta^kr)}|u_x|^3=\theta^{-2}\Y(z,\theta^kr)\\
&\le\theta^{ \alpha k-2} \Y(z,r)\\
&\le \theta^{-\alpha -2} \frac{\varrho}{r} \Y(z,r),
\end{align*}
which yields (\ref{decayest}).
\end{proof}

Combining this decay estimate with the nonlinear parabolic Poincar\'e inequality (Theorem \ref{PPI}) yields the following.

\begin{corollary}\label{localregOK}
Given $\alpha \in (0,3)$ there exist $\varepsilon_*>0$ and $R\in (0,1)$ such that if $r<R$ and
$$
\frac{1}{r^2}\int_{Q(z,r)}|u_x|^3<\varepsilon_*
 $$
  then
$$
\frac{1}{\varrho^5}\int_{Q(z,\varrho)}|u-u_{z,\varrho}|^3\le C \eps_* \left( \frac{\varrho}{r} \right)^{\alpha}
\qquad\mbox{for all}\quad \varrho\le r.
$$
\end{corollary}

We can now apply the parabolic Campanato Lemma (Lemma \ref{Campanato}) to yield our first conditional regularity result.

\begin{theorem}[Conditional regularity in terms of $u_x$]\label{localreg1}
  Given $\beta \in (0,1)$ there exist $\varepsilon_0>0$ and $R_0\in (0,1)$ such that if $r<R_0$ and
\begin{equation}\label{LRC1}
\frac{1}{r^2}\int_{Q(z,r)}|u_x|^3<\varepsilon_0
\end{equation}
then $u$ is $\beta$-H\"older continuous in $Q(z,r/2)$, with
 \eqnb\label{holder_continuity_def}
 |u(x_1,t_1)-u(x_2,t_2)|\le \frac{C}{r}\left( |x_1-x_2|+|t_1-t_2|^{1/4}\right)^\beta.
 \eqne
\end{theorem}

\begin{proof} Let $\varepsilon_0 \coloneqq \varepsilon_*/4$, $R_0 \coloneqq \min \{ 1, R \}$ and $r<R_0$, where $\varepsilon_*$, $R$ are from Corollary \ref{localregOK} applied with $\alpha =3\beta$. Then $Q(y,r/2)\subset Q(z,r)$ for every $y\in Q(z,r/2)$ and
 $$
 \frac{1}{(r/2)^2}\int_{Q(y,r/2)}|u_x|^3\le \frac{4}{r^2}\int_{Q(z,r)}|u_x|^3<4\eps_0 = \varepsilon_*.
 $$
 Thus Corollary \ref{localregOK} gives
 $$
\frac{1}{\varrho^5}\int_{Q(y,\varrho)}|u-u_{y,\varrho}|^3\,\d z\le   C \eps_* \left( \frac{\varrho}{r} \right)^{3\beta }
$$
for every $y\in Q(z,r/2)$ and every $0<\varrho\le r/2$. H\"older continuity of $u$ within $Q(z,r/2)$ now follows immediately from the Campanato Lemma (see Lemma \ref{Campanato}).\end{proof}

\subsection{Partial regularity I: box-counting dimension}
Bl\"omker \& Romito \cite{blomkerromito09} showed that if
$$
{\mathcal T}:=\{t\ge0:\ \|u\|_{H^1}\mbox{ is not essentially bounded in a neighbourhood of } t\}
$$
then $d_B({\mathcal T})\le 1/4$, where $d_B$ denotes the box-counting dimension (see their Remark 4.7 -- the proof is not actually given in their paper, but it follows easily from the estimates they obtain, using the argument from Robinson \& Sadowski \cite{rob_sad_2007}). Since $H^1(\T)\subset L^\infty(\T)$, it follows in particular that if
$$
{\mathcal T}_\infty:=\{t\ge0:\ \|u\|_{L^\infty}\mbox{ is not essentially bounded in a neighbourhood of } t\}
$$
then ${\mathcal T}_\infty\subseteq{\mathcal T}$, and so trivially $d_B({\mathcal T}_\infty)\le 1/4$. Since our singular set $S$ (recall \eqref{singset}) is a subset of ${\mathcal T}_\infty\times\T$, it follows from properties of the box-counting dimension that $d_B(S)\le 5/4$.

We now use the conditional regularity of the previous section to improve on this bound. We use the `Minkowski definition' of the box-counting dimension in our argument, namely
\eqnb\label{minkowski_def}
d_B (K) \coloneqq n - \liminf_{\delta \to 0^+} \frac{\log |K_\delta |}{\log \delta},\qquad K\subset \RR^n,
\eqne
where $K_\delta \coloneqq \{ y \colon \mathrm{dist}(y,K)<\delta \}$ denotes the $\delta$-neighbourhood of $K$. This formulation is one of a number of equivalent definitions of the box-counting dimension, see Proposition 2.4 in Falconer \cite{falconer}.

\begin{corollary}[Partial regularity I]\label{cor_bound_on_db} The space-time singular set $S$ (recall \eqref{singset}) satisfies $d_B(S\cap K) \leq 7/6$ for any compact set $K\subset \TT \times (0,\infty )$.
\end{corollary}
The reason for considering the intersection $S\cap K$ (instead of $S$) is technical, recall the comments following \eqref{singset}. In fact, it suffices to take $S\cap ( \TT\times [a,b] )$ (instead of $S\cap K$) for $0<a<b$.
\begin{proof}
Let $\eta\coloneqq \inf \{ t^{1/4}\colon (x,t)\in K\text{ for some }x \} $. Given $r\in (0 , \eta )$ let
\begin{equation*}
\begin{split}
M_r &\coloneqq \text{ maximal number of pairwise disjoint }r\text{-cylinders with centres in }S\cap K,\\
N_r &\coloneqq \text{ minimal number of }r\text{-cylinders with centres in }S\cap K\text{ needed to cover }S\cap K.
\end{split}
\end{equation*}
\emph{Step 1.} We show that $M_r \leq c r^{-5/3}$ for sufficiently small $r$.\\
%

Let $Q(z_1,r), \ldots , Q(z_{M_r},r)$ be a family of pairwise disjoint cylinders with centres $z_i\in S\cap K$ ($i=1,\ldots , M_r$). Note that the choice of sufficiently small $r$ above guarantees that these cylinders are contained within $\TT\times (0,\infty)$.
The conditional regularity result of Theorem \ref{localreg1} guarantees that for sufficiently small $r$
\[
\frac{1}{r^2}\int_{Q(z_i,r)}|u_x|^3\ge\varepsilon_0, \qquad i = 1,\ldots , M_r.
\]
Thus, since H\"older's inequality gives
\[
\int_{Q(z_i,r)}|u_x|^3\leq c \left(\int_{Q(z_i,r)}|u_x|^{10/3}\right)^{9/10}r^{1/2},
\]
we obtain, using \eqref{ux_is_in_L10/3},
\begin{equation*}
\begin{split}
c &> \int_0^T \int |u_x|^{10/3} \geq \sum_{i=1}^{M_r} \int_{Q(z_i,r)} |u_x|^{10/3} \\
&\geq c \sum_{i=1}^{M_r} \left( r^{-1/2}  \int_{Q(z_i,r)}|u_x|^3\right)^{10/9}\\
&\geq c  \sum_{i=1}^{M_r} r^{5/3} \varepsilon_0^{10/9}\\
&=c M_r r^{5/3}.
\end{split}
\end{equation*}
At this point it is interesting to note that since
$$
d_B(S\cap K)\leq \limsup_{r\to0}\frac{\log M_r}{-\log r}
$$
this bound on $M_r$ implies that $d_B(S\cap K)\le 5/3$ (as in the context of the Navier--Stokes equations, see \cite{rob_sad_2009}), but this does not improve on the bound $5/4$ mentioned above. However, unlike in the case of the Navier--Stokes equations, the use of the Minkowski definition \eqref{minkowski_def} gives a sharper bound (which is, in essence, a consequence of a dimensional analysis of the SGM; that is, roughly speaking, the dimension of time, $4$, is larger than the space dimension, $1$), which we show in the following steps.\\

\noindent\emph{Step 2.} We show that $N_{2r}\le M_r$ for all $r\in (0,\eta/2)$.\\

Let $\{ Q(z_i ,r) \}_{i=1}^{M_r}$ be a family of pairwise disjoint cylinders with centres $z_i=(x_i,t_i)\in S\cap K$. We will show that the family $\{ Q(z_i , 2 r )\}_{i=1}^{M_r}$ covers $S\cap K$, which proves the inequality above.
Indeed, suppose that this is not true, so that there exists $z_0=(x_0,t_0)\in S\cap K$ such that
\[
z_0 \not \in \bigcup_{i=1}^{M_r} Q(z_i, 2r )
\]
Then for each $i$
\[
|x_0-x_i |\geq 2r \qquad\mbox{or}\qquad
|t_0-t_i|\geq (2r)^4 > 2r^4,
\]
which shows that
\[
Q(z_0,r) \quad \text{ and } \quad Q(z_i,r)\quad \text{ are disjoint.}
\]
Thus $\{ Q(z_i,r) \}_{i=0}^{M_r}$ is a family of pairwise disjoint cylinders with centres in $S\cap K$, which contradicts the definition of $M_r$.\\

\noindent\emph{Step 3.} We deduce that $d_B(S\cap K) \leq 7/6$.\\

For $r<\min \{ 1, R_0 ,\eta/2\}$ let $\{ Q(z_i,r) \}_{i=1}^{N_{r}}$ be a family of pairwise disjoint $r$-cylinders which cover $S\cap K$ with centres $z_i=(x_i,t_i)\in S\cap K$. Note that
\eqnb\label{square_temp}
(S\cap K)_{r^4} \subset \bigcup_{i=1}^{N_r} Q(z_i,2r).
\eqne
Indeed, given $z=(x,t)\in (S\cap K)_{r^4}$ let $z_0\in S\cap K$ be such that $|z-z_0|<r^4$ and suppose that $z_0=(x_0,t_0)\in Q(z_i,r)$ for some $i\in \{ 1, \ldots , N_r \}$. Then
\begin{equation*}
\begin{split}
|x-x_i|& \leq |x-x_0|+|x_0-x_i|< r^4 +r<2r,\\
|t-t_i|& \leq |t-t_0|+|t_0-t_i|< 2r^4 < (2r)^4,
\end{split}
\end{equation*}
that is $z\in Q(z_i,2r)$, which shows \eqref{square_temp}. Therefore, using steps 1 and 2, we obtain
\[
|(S\cap K)_{r^4}|\leq N_r 2^7 r^5 \leq M_{r/2} 2^7 r^5 \leq c\, r^{10/3}.
\]
Letting $\delta \coloneqq r^4$ we obtain
\[
|(S\cap K)_\delta | \leq c\, \delta^{5/6}
\]
for all sufficiently small $\delta >0$. Thus
\[
\frac{\log |(S\cap K)_\delta |}{\log \delta } \geq \frac{\log c + \frac{5}{6}\log \delta }{\log \delta } \to \frac{5}{6} \qquad \text{ as } \delta \to 0^+,
\]
and so \eqref{minkowski_def} gives $d_B(S\cap K) \leq 7/6$.
\end{proof}
 Note that the above corollary gives in particular a similar bound on the Hausdorff dimension, $d_H (S\cap K) \leq 7/6$ (since $d_H(K) \leq d_B(K)$ for any compact $K$, by a property of the Hausdorff dimension, see, for example, Proposition 3.4 in Falconer \cite{falconer}), and so $|S\cap K|=0$ for any compact set $K$, which implies that $|S|=0$.

\section{The second conditional and partial regularity results}\label{sec_2nd_local_reg}

Here we show that there exists $\varepsilon_1>0$ such that any suitable weak solution $u$ is regular at $z=(x,t)$ whenever
\[
\limsup_{r\to 0} \frac{1}{r}\int_{Q(z,r)} u_{xx}^2 \leq \varepsilon_1 
\]
or
\[
  \limsup_{r\to 0}\left\{\mathrm{ess\,sup}_{s\in (t-r^4, t+r^4)}\frac{1}{r}\int_{B_r(x)}u(s)^2 \right\}<\varepsilon_1.
\]
Given $z=(x,t)$ we will write $B_r \coloneqq (x-r,x+r)$, $Q_r \coloneqq Q(z,r)$ and we will denote by $(u(s))_r \coloneqq (2r)^{-1} \int_{B_r} u(s)$ the mean of $u(s)$ over $B_r$. We will use the following quantities:
\begin{align*}
  A (r)&\coloneqq \mathrm{ess\,sup}_{s\in (t-r^4,t+r^4)} \frac{1}{r}\int_{B_r} u(s)^2\,\d x,\\
  \overline{A} (r)&\coloneqq \mathrm{ess\,sup}_{s\in (t-r^4,t+r^4)} \frac{1}{r}\int_{B_r}\left( u(s)-(u(s))_r \right)^2\,\d x,\\
  E(r)&\coloneqq \frac{1}{r}\int_{Q_r}u_{xx}^2\,\d z,\\
  W(r)&\coloneqq  \frac{1}{r^5}\int_{Q_r}|u|^3\,\d z,\\
  \Y(r)&\coloneqq \frac{1}{r^2}\int_{Q_r}|u_x|^3\,\d z.
\end{align*}
We note that each of the above quantities is invariant with respect to the scaling $u(x,t)\mapsto u(\lambda x,\lambda^4t)$. Furthermore $W$ and $\Y$ can be estimated in terms of $A$, $\overline{A}$ and $E$, which we make precise in the following lemma.
\begin{lemma}[Interpolation inequalities]
For every $r>0$
\begin{eqnarray}
W(r) &\leq& c A(r)^{11/8} E(r)^{1/8}+ c A(r)^{3/2} , \label{interp_W}\\
\Y(r) &\leq& c \overline{A}(r)^{5/8} E(r)^{7/8}. \label{interp_Y}
\end{eqnarray}
\end{lemma}
\begin{proof}
Due to scale-invariance we can assume that $r=1$. As for the estimate on $W(1)$ we write $\overline{u}(t)\coloneqq (u(t))_1$ and apply the decomposition
\[
u(x,t)= \left( u(x,t) - \overline{u} (t) \right) + \overline{u} (t) =: v(x,t)+ \overline{u}(t)
\]
Applying the $1D$ embedding $H^{1/6} \subset L^3 $ and using the fact that $v(t)$ has zero mean we can write (for each $t$)
\[
  \|v \|_{L^3(B_1)}^3\leq c \|v\|_{H^{1/6}(B_1)}^3 \leq c \|v\|_{\dot{H}^{1/6}(B_1)}^3,
\]
and so, by Sobolev interpolation,
\[
  \|v \|_{L^3(B_1)}^3\leq c \|v\|_{L^2(B_1)}^{11/4}\|\p_{xx}v \|_{L^2(B_1)}^{1/4}\leq  c \|u\|_{L^2(B_1)}^{11/4}\|\p_{xx}u \|_{L^2(B_1)}^{1/4},
\]
where we also used the fact that $\| v \|_{L^2 (B_1)} \leq 2 \| u \|_{L^2 (B_1)}$. Thus
\begin{equation*}
  \begin{split}
  \int_{-1}^1 \|v(t) \|_{L^3}^3 \d t&\leq c\left(\mathrm{ess\,sup}_{t\in (-1,1)}\|v(t)\|_{L^2}\right)^{11/4}\left(\int_{-1}^1 \|\p_{xx} u(t)\|_{L^2}^{1/4} \d t\right)\\
  &\leq c \, A(1)^{11/8}E(1)^{1/8}.
  \end{split}\end{equation*}
  We also have
  \begin{equation*}
  \begin{split}
  \int_{-1}^1 \|\overline{u}(t) \|_{L^3}^3 \d t&= c \int_{-1}^1 \left|  \int_{-1}^1 {u}(x,t) \,\d x \right|^3 \d t \leq c \int_{-1}^1 \left(  \int_{-1}^1 u(x,t)^2 \d x\right)^{3/2} \d t \\
  &\leq c \, A(1)^{3/2}.
  \end{split}
  \end{equation*}
  The last two inequalities show the required estimate on $W(1)$.\\
  
   As for the estimate on $\Y(1)$, we let $v(x,t)\coloneqq u(x,t) - (u(t))_1$ and write (for each $t$)
\begin{equation*} \begin{split}
  \|v_x\|_{L^3(B_1)}^2&\leq c\|v_x\|_{H^{1/6}(B_1)}^2\le c \sum_{k\in \ZZ} \left( 1+|k|^{1/3}\right) \left| \widehat{v_x} (k) \right|^2 \\
  &=c \sum_{k\ne 0} \left( k^2+|k|^{2+1/3}\right) \left| \widehat{v} (k) \right|^2 \leq c \sum_{k\ne 0} |k|^{2+1/3} \left| \widehat{v} (k) \right|^2\\
  &\le  c \| v \|_{\dot{H}^{7/6}(B_1)}^2 ,
  \end{split}
  \end{equation*}
  where $\hat{f} (k)$ denotes the $k$-th Fourier mode in the Fourier expansion of $f$ on $(-1,1)$. Applying Sobolev interpolation we obtain
  \[
  \|v_x\|_{L^3(B_1)}^3 \leq c \| v \|_{\dot{H}^{7/6}(B_1)}^3 \leq  c \| v \|_{L^2 (B_1)}^{5/4} \| \p_{xx} v \|_{L^2(B_1)}^{7/4},
  \]
  and thus
  \begin{align*}
  \Y(1)&=\int_{-1}^1\|v_x(t)\|_{L^3}^3 \d t \leq c\left(\mathrm{ess\,sup}_{t\in (-1,1)}\|v(t)\|_{L^2}\right)^{5/4}\int_{-1}^1 \|\p_{xx} v(t) \|_{L^2}^{7/4} \d t\\
   &\leq c \,\overline{A}(1)^{5/8}E(1)^{7/8}.
  \end{align*}
\end{proof}

We can now state the main theorem of this section.
\begin{theorem}[Conditional regularity in terms of $u_{xx}$]\label{thm_2nd_local_reg}
Given $\beta\in (0,1)$ there exists an $\varepsilon_1>0$ such that if
\be{Econd}
  \limsup_{r\to0}\frac{1}{r}\int_{Q(z,r)}u_{xx}^2<\varepsilon_1
\ee
  then $u$ is $\beta$-H\"older continuous (as in \eqref{holder_continuity_def}) in $Q(z,\rho)$ for some $\rho>0$. 
\end{theorem}

\begin{proof} The proof is inspired by Lin \cite{lin} and Kukavica \cite{kukavica_2009}. Without loss of generality we can assume that $z=(0,0)$. We will show that (\ref{Econd}) implies that
  \eqnb\label{y_will_be_small}
  \Y(r)\leq \varepsilon_0 \qquad \text{ for some }r\in (0,R_0),
  \eqne
which, in the light of Theorem \ref{localreg1}, proves the theorem.\\

\noindent\emph{Step 1.} We show the estimate
\[
\overline{A}(r/2)+E(r/2) \leq \frac{1}{2}\overline{A}(r)+c\left( E(r) + E(r)^{10}  \right)\qquad \text{ for any } r>0.
\]
Due to the scale invariance it is sufficient to take $r=1$. For brevity we will write $\overline{A}\coloneqq \overline{A}(1)$, $E\coloneqq E(1)$, $\Y\coloneqq \Y(1)$, as well as $B\coloneqq B_1$, $Q\coloneqq Q_1$. Let $\phi \in C_0^\infty (Q_{3/4};[0,1])$ be such that $\phi=1$ on $Q_{1/2}$ and $|\p_t \phi|, |\p_x^k \phi| \leq c$ for all $k\leq 4$.

Furthermore, let $\sigma \in C_0^\infty (B_1 ; [0,1])$ be such that $\sigma =1$ on $B_{3/4}$. 

We set 
\[
u^\sigma (t) \coloneqq \frac{\int_{B} u(t) \sigma \,\d x}{\int_B \sigma \,\d x}\quad \text{ and } \quad  [u]^\sigma  \coloneqq \frac{\int_{Q} u \sigma }{\int_Q \sigma }.
\]
In other words, recalling the notation \eqref{sigma_means_notation}, used in the proof of the Parabolic Poincar\'e inequality, we have $u^\sigma \equiv u^\sigma_1$, $[u]^\sigma \equiv [u]^\sigma_1$. Recall the Poincar\'e inequality \eqref{fact1_eq},
\eqnb\label{poincare_ineq_restated}
\int_{B} \left| u(t) - u^\sigma (t) \right|^3 \sigma  \leq c \int_{B} |u_x (t)|^3 \sigma ,
\eqne
and Corollary \ref{cor_to_poincare} (with $\eta = 1$),
\eqnb\label{claim_of_cor_to_ppi}
\int_Q \left| u-[u]^\sigma \right|^3\sigma\, \leq c \left( Y+ Y^2\right).
\eqne
Observe also that for almost every $t\in (-1,1)$ 
\eqnb\label{estimate_means_s_t_restated}
|u^\sigma (t) - [u]^\sigma |^{3}\leq c \left( Y+ Y^2\right),
\eqne
due to \eqref{est_of_the_means_other} with $\eta =1$.

The local energy inequality \eqref{LEI_alt_form} for $u-[u]^\sigma $ (recall the comments following \eqref{LEI_alt_form}) gives
\begin{align*}
  \overline{A}(1/2)&+E(1/2)\leq {2}\,\mathrm{ess\,sup}_{s\in (-2^{-4}, 2^{-4})}\int_{B_{1/2}}\left( u(s)- [u]^\sigma \right)^2+2\int_{Q_{1/2}}u_{xx}^2\\
  &\leq c\int_{Q_{3/4}} \left( u- [u]^\sigma \right)^2  +  c\int_{Q} \left( u_x^2+|u_x|^3\right) + c\int_{-(3/4)^{4}}^{(3/4)^{4}} \left| \int_{B_{3/4}} u_x^2\left( u- [u]^\sigma \right) \phi_{xx} \right|\\
  & \leq c \left( \int_{Q_{3/4}} \left| u- [u]^\sigma \right|^3 \right)^{2/3}  +  c(\Y^{2/3} +\Y ) +c \int_{Q_{3/4}} u_x^2\left| \left( u- u^\sigma  \right) \phi_{xx} \right|\\
  &\hspace{6.1cm}+c \int_{-1}^1 \left|\left( u^\sigma - [u]^\sigma \right) \int_{B} u_x^2 \phi_{xx} \right|,
\end{align*}
where we used the fact that $\int_B (f-(f)_1)^2 \leq \int_B (f-K)^2$ for any $K\in \RR$, $f\in L^2 (B)$ in the first line, the fact that $\mathrm{supp}\,\phi\subset Q_{3/4}$ in the second line, and H\"older's inequality and triangle inequality in the third line. Now, by applying \eqref{claim_of_cor_to_ppi} to the first of the resulting terms and integrating the last term by parts, we obtain
\begin{align*}
    \overline{A}(1/2)+E(1/2)&\leq c (\Y^{2/3} +\Y^{4/3} )+ c \int_{Q_{3/4}} u_x^2\left|  u- u^\sigma  \right|\\
    &\hspace{3cm}+c\int_{-1}^1 \left|\left( u^\sigma - [u]^\sigma \right) \int_{B} u_x u_{xx} \phi_{x} \right|\\
  & \leq c (\Y^{2/3} +\Y^{4/3} )+ c \Y^{2/3} \left( \int_{Q} \left|  u- u^\sigma  \right|^3 \sigma \right)^{1/3}\\
  &\hspace{3cm}+ c \left( \mathrm{ess\, sup}_{s\in(-1,1)} \left| u^\sigma (s)- [u]^\sigma \right|  \right) \int_{Q}  \left| u_x u_{xx}  \right|,
\end{align*}
where we also applied H\"older's inequality and used the fact that $\sigma =1$ on $Q_{3/4}$ in the second line. Finally, applying \eqref{poincare_ineq_restated}, \eqref{estimate_means_s_t_restated}, the Cauchy-Schwarz inequality and H\"older's inequality gives
\begin{align*}  
   \overline{A}(1/2)+E(1/2)  &\leq c (\Y^{2/3} +\Y^{4/3} )+  c \left( \Y^{1/3} + \Y^{2/3} \right) \Y^{2/3} E^{1/2} \\
  &= c (\Y^{2/3} +\Y^{4/3} )+  c E^{1/2} \left( \Y + \Y^{4/3} \right) \\
  &\leq c \left( \overline{A}^{5/12}E^{7/12}+ \overline{A}^{5/6}E^{7/6} + \overline{A}^{5/8}E^{11/8} + \overline{A}^{5/6} E^{5/3}  \right)\\
  &\leq \frac{1}{2}\overline{A} + c \left( E + E^{10}\right),
  \end{align*}
  as required, where we also used the interpolation inequality \eqref{interp_Y} in the third line, and Young's inequality $ab\leq \delta a^p + C_\delta b^q$ (where $1/p+1/q=1$ and sufficiently small $\delta >0$) in the last line.\\
  
\noindent\emph{Step 3.} We show \eqref{y_will_be_small}.\\

Let $\varepsilon_1 >0$ be small enough that
\[
c\left( \varepsilon_1  +\varepsilon_1^{10}  \right) \leq \frac{1}{4}\varepsilon_0^{2/3}.
\]
By assumption there exists $r_0$ such that $E(r)<\varepsilon_1$ for $r\in (0,r_0]$. From Step 2
\[
\overline{A}(r/2)+E(r/2) \leq \frac{1}{2} \overline{A}(r) + \frac{1}{4}\varepsilon_0^{2/3},\quad r\in (0,r_0],
\]
and iterating this inequality $k$ times we obtain
\[
\overline{A}(2^{-k}r_0)+E(2^{-k}r_0 )\leq 2^{-k} \overline{A}(r_0) + \frac{1}{4}\varepsilon_0^{2/3} \sum_{j=0}^{k-1} 2^{-j} \leq 2^{-k} \overline{A}(r_0) + \frac{1}{2}\varepsilon_0^{2/3} .
\]
Thus for sufficiently large $k$
\[
\overline{A}(2^{-k}r_0)+E(2^{-k}r_0 )\leq \varepsilon_0^{2/3},
\]
and so interpolation inequality \eqref{interp_Y} gives
\[
\Y(2^{-k}r_0) \leq \overline{A}(2^{-k}r_0)^{5/8} E(2^{-k}r_0)^{7/8} \leq \varepsilon_0^{5/12} \varepsilon_0^{7/12} = \varepsilon_0   ,
\]
as required.
\end{proof}
\begin{corollary}[Conditional regularity in terms of $\mathrm{ess\,sup}_t \int_{B_r} u(t)^2$]
There exists an $\varepsilon_2>0$ such that if
\be{Acond}
  \limsup_{r\to 0}\left\{\mathrm{ess\,sup}_{s\in (t-r^4, t+r^4)}\frac{1}{r}\int_{B_r(x)}u(s)^2 \right\}<\varepsilon_2
\ee
  then $u$ is $\beta$-H\"older continuous in $Q(z,\rho)$ for some $\rho>0$.

\end{corollary}
\begin{proof}
The claim follows by replacing the estimate from Step 1 above by 
\eqnb\label{iteration_1}
A(r/2)+E(r/2) \leq \frac{1}{2} E(r) + c\left(A(r)+ A(r)^5\right),
\eqne
whose proof we defer for a moment. Indeed, then one can choose $\varepsilon_2>0$ sufficiently small such that $c\left( \varepsilon_2 + \varepsilon_2^5 \right)\leq  \varepsilon_0^{2/3}/4$ and the claim follows as in Step 3 above by noting that $\overline{A}\leq A$. We now verify \eqref{iteration_1}, where we assume that $r=1$, as before. Using the local energy inequality \eqref{LEI_alt_form} we obtain 
\begin{align*}
  A(1/2)&+E(1/2)\leq c \int_Q \left( u^2 + u_x^2 +|u_x|^3 + u_x^2 |u|\right)\\
  &\leq c \left(   A + \Y^{2/3} + \Y + \Y^{2/3} W^{1/3} \right)\\
  &\leq c \left(   A + A^{5/12}E^{7/12}  + A^{5/8} E^{7/8} + A^{43/24} E^{17/24} + A^{23/12} E^{7/12}  \right)\\
  &\leq \frac{1}{2} E + c \left( A + A^5\right),
  \end{align*}
  as required, where we used H\"older's inequality in the second line, the interpolation inequalities \eqref{interp_W}, \eqref{interp_Y} (together with a fact that $\overline{A}\leq A$) in the third line, and Young's inequality $ab\leq \delta a^p + c_\delta b^q$ (where $1/p+1/q=1$ and $\delta >0$ is chosen sufficiently small).
\end{proof}

Using Theorem \ref{thm_2nd_local_reg} we can obtain improved bounds on the dimension of the singular set in terms of the (parabolic) Hausdorff measure. For a set $X\subset\R\times\R$ and $k\ge0$ let
\eqnb\label{def_of_Pk}
P^k(X) \coloneqq \lim_{\delta\rightarrow0^+}P_\delta^k(X)
\eqne
denote the $k$-dimensional parabolic Hausdorff measure, where
$$
P_\delta^k(X)\coloneqq \inf\left\{\sum_{i=1}^\infty r_i^k:\ X\subset\bigcup_i Q_{r_i}:\ r_i<\delta\right\},
$$
and $Q_{r_i} = Q_{r_i}(x,t)$ is a $r_i$-cylinder, $i\geq 1$. Observe that $P^1(X)=0$ if and only if for every $\delta>0$ the set $X$ can be covered by a collection $\{Q_{r_i}\}$ such that $\sum_ir_i<\delta$.
\begin{corollary}[Partial regularity II]\label{cor_bound_on_dH}
The singular set $S$ of a suitable weak solution of \eqref{BR} satisfies $\mathcal{P}^1(S)=0$.
\end{corollary}
Note that this in particular gives $d_H(S)\leq 1$ (since $\mathcal{H}^1 (S) \leq c \mathcal{P}^1 (S)$, where $\mathcal{H}^1$ denotes the $1$-dimensional Hausdorff measure). 

We will need the Vitali Covering Lemma in the following form: given a family of parabolic cylinders $Q_r(x,t)$, there exists a countable (or finite) disjoint subfamily $\{Q_{r_i}(x_i,t_i)\}$ such that for any cylinder $Q_r(x,t)$ in the original family there exists an $i$ such that $Q_r(x,t)\subset Q_{5r_i}(x_i,t_i)$. (For a proof see Caffarelli, Kohn \& Nirenberg (1982).)

\begin{proof} Fix $\delta >0$ and let $V$ be an open set containing $S$ such that
\[
\frac{5}{\varepsilon_1 } \int_V u_{xx}^2 \leq \delta .
\]
Such $V$ exists since $u_{xx}\in L^2 (\TT \times (0,T))$ (recall \eqref{weak_sol_regularity}) and since $|S|=0$ (see the comments preceding this section).
For each $(x,t)\in S$, choose $r\in (0,\delta )$ such that $Q_{r/5}(x,t)\subset V$ and
$$
\frac{5}{r}\int_{Q_{r/5}(x,t)}u_{xx}^2>\varepsilon_1.
$$
Such a choice is possible, for otherwise the point $(x,t)$ would be regular due to Theorem \ref{thm_2nd_local_reg}. We now use the Vitali Covering Lemma to extract a countable (or finite) disjoint subcollection of these cylinders $\{Q_{r_i/5}(x_i,t_i)\}$ such that the singular set $S$ is still covered by $\{Q_{r_i}(x_i,t_i)\}$. Then
\[
\sum_i r_i \leq \frac{5}{\varepsilon_1} \sum_i \int_{Q_{r_i/5}(x_i,t_i)}u_{xx}^2 \leq \frac{5}{\varepsilon_1}  \int_V u_{xx}^2 \leq \delta,
\]
as required.
\end{proof}

\section{Conclusion and further discussion}

We have proved two conditional regularity results, and as a consequence two bounds on singular space-time set for the SGM:
$$
d_B(S\cap K)\le 7/6\qquad\mbox{and}\qquad P^1(S)=0,
$$
for any compact $K\subset \TT\times (0,\infty )$.
As with the Navier--Stokes equations, there is a gap here between the box-counting and Hausdorff dimensions; as with the NSE, it is an open question whether these dimension estimates can be equalised.

In this context, it would be interesting to adapt the constructions due to Scheffer \cite{scheffer_a_sol,scheffer_nearly}, see also O\.za\'nski \cite{Scheffer_stuff,ozanski_any_energy_profile}) of solutions of the weak form of the `Navier--Stokes inequality' that have a space-time singular set of Hausdorff dimension $\gamma$ for any $\gamma\in(0,1)$ to the SGM. This seems difficult, since the constructions make use of (i) the three-dimensional nature of the fluid flow and (ii) the pressure function plays a fundamental role in amplifying the magnitude of the velocity.

There are some outstanding conditional regularity problems for the SGM: one is to prove a local version of the $L^{8/(2\alpha-1)}(0,T;H^\alpha)$ regularity condition (this result is only known in a global form, see introduction); and the other to prove the same for $u\in L^\infty_tL^\infty_x$, both globally and locally. In particular the first would imply that the complement of our `singular set' $S$ really does consist of points in a neighbourhood of which $u$ is regular in space.
\section*{Acknowledgements}
WSO is supported by EPSRC as part of the MASDOC DTC at the University of Warwick, Grant No. EP/HO23364/1; JCR was partially supported by an EPSRC Leadership Fellowship EP/G007470/1.
\appendix

\section{A general (parabolic) Campanato Lemma}

Let $x\in\R^{n_1}$, $y\in\R^{n_2}$, and write $z=(x,y)$. Let $Q_r(z)$ denote the set $B_r(x)\times B_{r^\alpha}(y)$ with volume $Vr^n$, where $n=n_1+\alpha n_2$, $\alpha \in \NN$.

For any $f\in L^1(Q_r(z))$, define
$$
f_{z,r}=\fint_{Q_r(z)}f(y)\,\d y=\frac{1}{r^nV}\int_{Q_r(z)}f(y)\,\d y.
$$

\begin{lemma}[Comparison of averages]\label{lem:averages}
If $f\in L^1(Q_r(z))$ then
\be{star}
|f_{z,\theta r}-f_{z,r}|\le\theta^{-n}\left(\fint_{Q_r(z)}|f(y)-f_{z,r}|^p\,\d y\right)^{1/p}.
\ee
\end{lemma}

\begin{proof}
We have
\begin{align*}
|f_{z,\theta r}&-f_{z,r}|^p=\left|\frac{1}{(\theta r)^nV}\int_{Q_{\theta r}(z)}f(y)-f_{z,r}\,\d y\right|^p\\
&\le\frac{1}{(\theta r)^{pn}V^p}\left(\int_{Q_{\theta r}(z)}|f(y)-f_{z,r}|\,\d y\right)^p\\
&\le\frac{1}{(\theta r)^{pn}V^p}\left(\int_{Q_r(z)}|f(y)-f_{z,r}|\,\d y\right)^p\\
&\le\frac{1}{(\theta r)^{pn}V^p}\left[\left(\int_{Q_r(z)}|f(y)-f_{z,r}|^p\,\d y\right)^{1/p}\left(\int_{Q_r(z)}\,\d y\right)^{(p-1)/p}\right]^p\\
&=\frac{1}{(\theta r)^{pn}V^p}\left(\int_{Q_r(z)}|f(y)-f_{z,r}|^p\,\d y\right)(r^nV)^{p-1}\\
&=\frac{1}{\theta^{pn}r^nV}\left(\int_{Q_r(z)}|f(y)-f_{z,r}|^p\,\d y\right)\\
&=\theta^{-pn}\fint_{Q_r(z)}|f(y)-f_{z,r}|^p\,\d y,
\end{align*}
which yields (\ref{star}).
\end{proof}

\begin{lemma}[Campanato]\label{Campanato}
  Let $R\in(0,1)$, $f\in L^1(Q_R(0))$ and suppose that there exist positive constants $\beta\in(0,1]$, $M>0$, such that
  \be{CampA}
  \left(\fint_{Q_r(z)}|f(y)-f_{z,r}|^p\,\d y\right)^{1/p}\le Mr^{\beta}
  \ee
  for any $z\in Q_{R/2}(0)$ and any $r\in(0,R/2)$. Then $f$ is H\"older continuous in $Q_{R/2}(0)$: for any $z,w\in Q_{R/2}(0)$, $z=(x,t)$, $w=(y,s)$,
  \be{fHolder}
  |f(x,t)-f(y,s)|\le cM(|x-y|+|t-s|^{1/\alpha})^\beta.
  \ee
\end{lemma}

\begin{proof}
Choose $z\in Q_{R/2}(0)$ and $r<R/2$. Using Lemma \ref{lem:averages} we can compare $f_{z,r/2}$ with $f_{z,r}$:
$$
|f_{z,r/2}-f_{z,r}|\le 4^n\left(\fint_{Q_r(z)}|f(y)-f_{z,r}|^p\,\d t\right)^{1/p}\le 4^nMr^\beta.
$$
Now consider $f_{z,2^{-k}}-f_{z,r}$. Since
$$
f_{z,r2^{-k}}-f_{z,r}=\sum_{j=1}^{k} f_{z,r2^{-k}}-f_{z,r2^{-(k-1)}},
$$
it follows that
\be{avav}
\left|f_{z,r2^{-k}}-f_{z,r}\right|\le\sum_{j=1}^k cMr^\beta 2^{-(j-1)\beta}\le\sum_{j=0}^\infty cMr^\beta 2^{-j\beta}=cMr^\beta.
\ee
This shows that $f_{z,r2^{-k}}$ forms a Cauchy sequence, and hence the averages converge for every $z\in Q_{r/2}(0)$. By the Lebesgue Differentiation Theorem, these averages converge to $f(z)$ for almost every $z$, and so if we let $k\rightarrow\infty$ in (\ref{avav}) we obtain an estimate for the difference between $f(z)$ and its average,
$$
|f(z)-f_{z,r}|\le c_1Mr^\beta.
$$

Now take another point $\zeta\in Q_{R/2}(0)$; we compare $f_{z,r}$ with $f_{\zeta,2r}$:
\begin{align*}
|f_{z,r}-f_{\zeta,2r}|^p&=\left|\frac{1}{r^nV}\int_{Q_r(z)}f(y)-f_{\zeta,2r}\,\d y\right|^p\\
&\le\frac{1}{r^nV}\int_{Q_r(z)}|f(y)-f_{\zeta,2r}|^p\,\d y,
\end{align*}
arguing as before. Now if $z=(x,t)$ and $\zeta=(\xi,s)$, choose
$$
r=|x-\xi|+|t-s|^{1/\alpha};
$$
then $Q_r(x,t)\subset Q_{2r}(\xi,s)$, and so we can enlarge the domain of integration to obtain
\begin{align*}
|f_{z,r}-f_{\zeta,2r}|^p&\le \frac{2^n}{(2r)^nV}\int_{Q_{2r}(\zeta)}|f(y)-f_{\zeta,2r}|^p\,\d y\\
&=2^n\fint_{Q_{2r}(\zeta)}|f(y)-f_{\zeta,2r}|^p\,\d y\\
&\le 2^nM^p(2r)^{p\beta},
\end{align*}
i.e.
$$
|f_{z,r}-f_{\zeta,2r}|\le 2^{p+(n/p)}Mr^\beta.
$$

So now (still with $r$ chosen as above)
\begin{align*}
|f(z)-f(\zeta)|&\le|f(z)-f_{z,r}|+|f_{z,r}-f_{\zeta,2r}|+|f_{\zeta,2r}-f(\zeta)|\\
&\le c_1Mr^\beta+c_2Mr^\beta+c_1M(2r)^\beta\\
&=cM(|x-y|+|t-s|^{1/\alpha})^\beta,
\end{align*}
which is (\ref{fHolder}).
\end{proof}

\bibliographystyle{plain}
\bibliography{literature}
\end{document}